\documentclass[11pt,a4paper,reqno]{amsart}
\usepackage{amsmath}
\usepackage{amsthm}
\usepackage{amstext,amsopn,amssymb,amsxtra}
\usepackage{mathtools}
\usepackage{hyperref}
\usepackage{pdfsync}
\usepackage{color}
\usepackage{tikz}

\usepackage{latexsym,amssymb, amsfonts, graphicx}
\usepackage{amscd, amsrefs}

\newtheorem{thm}{Theorem}[section]

\newtheorem{lem}[thm]{Lemma}
\newtheorem{prop}[thm]{Proposition}
\theoremstyle{definition}

\theoremstyle{remark}
\newtheorem{rmk}[thm]{Remark}

\numberwithin{equation}{section}

\DeclareMathOperator{\supp}{supp}
\DeclareMathOperator{\dist}{dist}

\title[Restriction estimate for hyperbolic surfaces]{Improved restriction estimate for hyperbolic surfaces in $\mathbb R^3$}
\author{Chu-Hee Cho and Jungjin Lee}

\address{School of Mathematical Sciences, Seoul National University, Seoul 08826, Republic of Korea}
\email{akilus@snu.ac.kr}

\address{Department of Mathematical Sciences, School of Natural Science, Ulsan National Institute of Science and Technology, UNIST-gil 50, Ulsan 44919,
	Republic of Korea}
\email{jungjinlee@unist.ac.kr}

\date{\today}



\begin{document}

\begin{abstract}
Recently, L. Guth improved the restriction estimate for the surfaces with strictly positive Gaussian curvature in $\mathbb R^3$. In this paper we extend his restriction estimate to the surfaces with strictly negative Gaussian curvature.  	
\end{abstract}
\maketitle
	
\section{Introduction}
Let $S$ be a smooth compact hypersurface with boundary in $\mathbb R^d$, which has a surface measure $d\sigma$. The Fourier transform of the measure $fd\sigma$ is written as
\[ 
\widehat{f d\sigma}(x) = \int_S e^{2 \pi ix \cdot w} f(w) d\sigma(w). 
\]
The \textit{restriction problem} posed by Stein \cite{stein1979some} is to find $(p,q)$ for which the \textit{adjoint restriction estimate}
\begin{equation} \label{ReaEst}
\| \widehat{f d\sigma} \|_{q}  \le C \|f\|_{L^p(S)}
\end{equation}
holds for all $f \in C_c^{\infty}(S)$, where the constant $C$ may depend on $p,q,d, S$ but not on $f$. 
This problem is connected to questions about the convergence of Fourier summation methods such as the Bochner-Riesz conjecture and local smoothing conjecture. Also, there is a fundamental relation between the restriction problem and the Kakeya problem. Moreover, the restriction problem is associated with the analysis of linear PDE such as the Helmholtz equation, Schr\"odinger equation, wave equation and the Korteweg-de Vries equation. See \cites{tao2001rotating, bourgain2000harmonic }.

For several decades, a fair amount of work was devoted to this problem (particularly when $S$ is an elliptic surface such as the unit sphere and paraboloid). After Bourgain \cite{bourga1991besicovitch} combined a multiscale analysis approach with his Kakeya estimate, Bourgain's methods were developed over the years; see \cites{moyua1999restriction, tao1998bilinear, tao2000bilinearI}. Especially, from the analysis of $L^2$ bilinear variants of the problem, Wolff \cite{wolff2001sharp} and Tao \cite{tao2003sharp} obtained the \textit{$L^2$ bilinear restriction theorem} for the cones and  paraboloids respectively, which made a significant progress on the restriction problems. On the other hand,  Bennet, Carbery and Tao \cite{bennett2006multilinear}, using the heat-flow method, obtained the \textit{multilinear Kakeya theorem} and  the \textit{multilinear restriction theorem}. 
(Later, Guth \cites{guth2010endpoint, guth2015short} gave an alternative proof of the multilinear Kakeya theorem.) 
After several years, Bourgain and Guth \cite{bourgain2011bounds} found a new way to apply the multilinear restriction theorem to the restriction problem, and they obtained some improvements. Recently, Guth \cite{guth2015restriction} further developed it in $\mathbb R^3$ by adapting the \textit{polynomial partitioning}. (It is a method that has brought some important results about overlapping lines in incidence geometry; see \cites{guth2010erdos, kaplan2012simple}.)  

In \cite{guth2015restriction}, Guth considered the restriction estimate for surfaces with strictly positive Gaussian curvature.
The aim of this paper is to extend Guth's restriction estimate to the case of quadratic surfaces with strictly negative Gaussian curvature in $\mathbb R^{3}$. 
The following is our main result.

\begin{thm} \label{thm:MM}
Let $S$ be a compact quadratic surface with strictly negative Gaussian curvature in $\mathbb R^3$. Then, for $p > 3.25$ and $p=q$, the estimate \eqref{ReaEst} is valid.
\end{thm}

Stein \cite{stein1986oscillatory} verified that the  estimate \eqref{ReaEst} holds for $q \ge 4$ and $\frac{2}{q} \le 1- \frac{1}{p}$. 
The best previously known result due to Lee \cite{lee2006bilinear} and Vargas \cite{vargas2005restriction} was $q> 10/3$ and $\frac{2}{q} < 1- \frac{1}{p}$. 
By interpolating Theorem \ref{thm:MM} with the previous result, the $(p,q)$-range is extended to $q>3.25$, $\frac{3}{q} < 1- \frac{1}{4p} $ and $\frac{2}{q} < 1- \frac{1}{p}$.

Our proof is based on Guth's arguments in \cite{guth2015restriction}. The key ingredients in his arguments are a \textit{broad function, polynomial partitioning, induction} and \textit{bilinear estimate}. 
Roughly speaking, the polynomial partitioning and induction are used to reduce a 3-dimensional restriction problem to an essentially 2-dimensional one. The broad function is exploited for a bilinear approach to the derived 2-dimensional problem.
We will modify the definition of broad function and the related bilinear estimates. 
As mentioned in \cite{lee2006bilinear} and \cite{vargas2005restriction}, we need a stronger separation condition to obtain bilinear restriction estimates for hyperbolic surfaces than that for elliptic ones. Accordingly, our broad function will be defined to involve such strong separation condition. Then, it is possible to have the same bilinear estimates as in \cite{guth2015restriction}.

The paper is organized as follows. In next section, we prepare the proof of our result by giving an elementary proposition about a wave packet decomposition. In section 3, we define a broad function, and reduce a Fourier restricted function to its broad function. In section 4, we prove the main part. We trim down the problem by using a polynomial partitioning and induction arguments, and then we bilinearly approach the remaining part. 

Throughout the paper we use $C$ to denote positive constants $\ge 1$ which may be different at each occurrence. We denote $A \lesssim B$ or $A=O(B)$ to mean $A \le CB$, and $A \sim B$ to mean $C^{-1}B \le A \le CB$. We denote the number of members of a set $A$ by $\# A$.

\section{Wave packet decomposition}
In this section we recall a wave packet decomposition which has been a fundamental tool in restriction problems.

By a suitable translation and linear transformation we may set  $S$ as the hyperbolic paraboloid defined by
\[
S = \{ (\xi_1, \xi_2, \xi_1 \xi_2) \in \mathbb R^3:~ (\xi_1, \xi_2) \in D(1) \},
\]
where $D(1)$ is the unit square centered at the origin. Let us define the extension operator $Ef$ by
\[
E f(x) = \int_S e^{2 \pi i x \cdot \xi} f(\xi) d\sigma(\xi)
\sim \int_{D(1)} e^{2 \pi i(x_1\xi_1+x_2\xi_2+x_3 \xi_1\xi_2)} \tilde f(\xi_1,\xi_2) d \xi_1 d\xi_2
\]
where $\tilde f(\xi_1,\xi_2) = f(\xi_1,\xi_2,\xi_1\xi_2) \frac{d\sigma(\xi_1,\xi_2)}{d\xi_1 d\xi_2}$.

We decompose $S$ into caps \( \Omega \) of diameter \( R^{-1/2} \). 
Let \( n(\Omega) \) be the unit normal vector to \( S \) at the center of $\Omega$. Let \( \delta>0 \) be a small parameter. For each cap \( \Omega \), we define \( \mathbb T(\Omega) \) to be the set of cylindrical tubes $T$ of radius \( R^{1/2 + \delta} \) which are parallel to \(  n(\Omega) \) and cover  a ball $B(R)$ of radius $R > 1$ with finite overlap. 
If $T \in \mathbb T(\Omega)$ then $v(T)$ indicates $n(\Omega)$, and  $\omega(T)$ denotes the center of $\Omega$. We define \( \mathbb T = \bigcup_{\Omega} \mathbb T(\Omega) \). 

We use the following standard wave packet decomposition.
This is a simple modification of Proposition 2.6 in \cite{guth2015restriction}. (We can find a similar decomposition in \cite{lee2006bilinear}*{Lemma 2.2} and in \cite{vargas2005restriction}*{section 3}.) 
\begin{prop}[Wave packet decomposition] \label{prop:wavepack}
Let \( R \gg 1 \) and let $B(R)$ be a ball of radius $R$. If \( f \in L^2(S) \), then for each tube \( T \in \mathbb T \) there exists a function \( f_T \) satisfies the following conditions :
\begin{enumerate}
\item 
If \( T \in \mathbb T(\Omega) \) then \( \supp f_T \subset 3\Omega. \)
\item 
If \( x \in B(R) \setminus T \)
then \( |E f_T(x)| \le R^{-1000}\|f\|_{L^2(S)} \).
\item 
For any \( x \in B(R) \), \( | E f(x) - \sum_{T \in \mathbb T} E f_T(x)| \le R^{-1000}\|f\|_{L^2(S)} .\)
\item 
If \( T_1, T_2 \in \mathbb T(\Omega) \) and \( T_1,T_2 \) are disjoint, then \( \int f_{T_1} \bar f_{T_2} \le R^{-1000}\int_{\Omega} |f|^2 \).
\item 
\( \sum_{T \in \mathbb T(\Omega)} \int_S |f_T|^2 \lesssim \int_\Omega |f|^2\).
\item
Let $\tau \subset S$ be a cap of radius $> 10R^{-1/2}$ and $f_{\tau} := f \chi_\tau$. Then for any $\mathbb T' \subset \mathbb T$ and any $\omega \in S$,
\[
\Big\| \sum_{T \in \mathbb T': \omega(T) \in \tau} f_T \Big\|_{L^2(B(\omega,R^{-1/2}) \cap S)} \lesssim \| f_{\tau} \|_{L^2(10B(\omega,R^{-1/2}) \cap S)}.
\]
\end{enumerate}
\end{prop}

\noindent The proof will be given in Appendix.
	
\section{Reduction and the broad function}
In this section we reduce the restriction estimate to a problem of obtaining good localized estimates for some regularized (adjoint) restriction operator.  

As in \cite{bourga1991besicovitch},
by the Stein-Nikishin factorization theorem, it suffices to show \eqref{ReaEst} for $q > 3.25$ and $p=\infty$. Furthermore, by Tao's \( \epsilon \)-removal lemma it is reduced to showing the following:

\begin{thm}\label{main theorem}
Let $p_0 = 3.25$. For any $R \ge 1$,  the estimate
\begin{equation}\label{main estimate}
\|E f \|_{L^{p_0}(B(R))} \le C_{\epsilon} R^\epsilon \|f\|_{L^\infty(S)}
\end{equation}
is valid for all $f$ on $S$, all $0 <\epsilon \ll 1$ and all ball $B(R)$ of radius $R$. 
\end{thm}
By translation invariance we may assume that $B(R)$ is centered at the origin.

Fix $R \gg 1$; in the case $R \sim 1$, it is easy to see \eqref{main estimate}. 
First, we take a large dyadic number \( K=K(\epsilon) \) with $\lim_{\epsilon \rightarrow 0} K(\epsilon) = \infty$ (we may set $K \sim e^{\epsilon^{-10}}$).
We divide $D(1)$ into $K^2$ squares $\bar \tau$ of sidelength $K^{-1}$ whose sides are parallel to standard unit vectors $e_1$ and $e_2$.
Let $L_{\parallel e_1}$ denote the $K$ strips of width $K^{-1}$ such that their center lines are parallel to $e_1 \in \mathbb R^{2}$ and they are composed of the squares $\bar \tau$.   $L_{\parallel e_2}$ are similar strips but their center lines are parallel to $e_2$.  Let $\tau := \{  (\xi_1,\xi_2,\xi_1\xi_2): (\xi_1,\xi_2) \in \bar\tau \}$. Then the surface $S$ is covered by the \( K^2 \) caps \( \tau \) of diameter \( \sim K^{-1} \). 
Set \( f_\tau = \chi_\tau f\). 

For $\alpha \in (0,1)$, we define an \textit{$\alpha$-broad point} of $Ef$ to be the point $x$ at which 
\[
\max_{\tau} |Ef_\tau(x)| + \max_{L = L_{\parallel e_1} \text{ or } L_{\parallel e_2 }} \Big| \sum_{\tau: \bar\tau \subset L }Ef_\tau(x) \Big|  \le \alpha |Ef(x)|.
\]
If $A_{\alpha}$ is the set of all $\alpha$-broad points of $Ef$, then we define an \textit{$\alpha$-broad function} \( \mathbf{B}_\alpha[Ef]\) by
$
\mathbf{B}_\alpha[Ef](x) = Ef(x)\chi_{A_\alpha}(x).
$
Then for given \( x \in B(R) \), there exist $\tau$ and $L$ such that
\[
|E f(x)| \le |\mathbf{B}_{\alpha} [Ef](x)| + \alpha^{-1} \Big( |Ef_\tau(x)| + \Big|\sum_{\tau: \bar\tau \subset L } Ef_{\tau}(x)\Big|\Big).
\]
From this we have that for any $x \in B(R)$,
\[
|E f(x)|^{p_0} \lesssim |\mathbf{B}_{\alpha} [Ef](x)|^{p_0} + \alpha^{-p_0} \Big( \sum_{\tau} |Ef_\tau(x)|^{p_0} + \sum_{L} \Big|\sum_{\tau: \bar\tau \subset L } Ef_{\tau}(x)\Big|^{p_0} \Big).
\]
By integrating over $B(R)$, 
\begin{equation} \label{eqn:star}
\begin{split}
\int_{B(R)} 
|E f(x)|^{p_0} &\lesssim \int_{B(R)} |\mathbf{B}_{\alpha} [Ef](x)|^{p_0} \\
&\quad + \alpha^{-p_0} \Big(   \sum_{\tau} \int_{B(R)} |Ef_{\tau}(x)|^{p_0}  + \sum_{L = L_{\parallel e_1} \text{ or } L_{\parallel e_2 }} \int_{B(R)} |\sum_{\tau: \bar\tau \subset L } Ef_{\tau}(x)|^{p_0} \Big).
\end{split}
\end{equation}

We first deal with the summation parts of the above inequality.
For this we use an inductive argument on $R$; we assume that for any $1 \le r \le R/2$ , the estimate \eqref{main estimate} holds for all $f$, all $0< \epsilon \ll 1$ and all balls $B(R)$.

By using the induction hypothesis we can prove the following estimates by scaling.
\begin{lem}
\begin{align} \label{eqn:tau}
\| E f_\tau \|_{L^{p_0}(B_{R})} &\le CC_\epsilon K^{-2+\frac{4}{p_0}} R^{\epsilon} \| f_\tau \|_{L^\infty(S)},\\
\label{eqn:sumtau}
\Big\| \sum_{\tau: \bar\tau \subset L } E f_\tau\Big\|_{L^{p_0}(B_{R})} &\le CC_\epsilon K^{-1+\frac{2}{p_0}} R^{\epsilon} \Big\| \sum_{\tau: \bar\tau \subset L } f_\tau \Big\|_{L^\infty(S)}.
\end{align}
\end{lem}
\begin{proof}
We first show \eqref{eqn:tau}.  
By translation we may assume that $\bar \tau$ is centered at the origin. By abuse of notation, we use $f$ instead of $\tilde f$. Then $f_\tau$ is supported in the square of sidelength $K^{-1}$ with center at the origin. By scaling $(\xi_1,\xi_2) \to (K^{-1}\xi_1, K^{-1}\xi_2)$,
\[ 
Ef_\tau(x_1,x_2,x_3)  = K^{-2} [Ef_\tau^K](K^{-1}x_1,K^{-1}x_2, K^{-2}x_3)
\]
where $f_\tau^{K} = f_\tau(K^{-1}\xi_1,K^{-1}\xi_2)$.
Note that $f_\tau^K$ is supported in the unit square.
By a change of variables, 
\[
\|  [Ef_\tau^K](K^{-1}\cdot,K^{-1}\cdot, K^{-2}\cdot) \|_{L^{p_0}(B(R))} = K^{4/p_0}\|  Ef_\tau^K  \|_{L^{p_0}(T)} 
\]
where $T$ is a tube of dimensions $R/K \times R/K \times R/K^2$.
From the above equations, we have
\[ 
\| Ef_\tau \|_{L^{p_0}(B(R))} \le  K^{-2+\frac{4}{p_0}} \|  Ef_\tau^K  \|_{L^{p_0}(T)}.
\]
Since $ \|  Ef_\tau^K  \|_{L^{p_0}(T)} \le \|  Ef_\tau^K  \|_{L^{p_0}(B(R/2))}$, we can apply the induction hypothesis. Thus,
\[ 
\| Ef_\tau\|_{L^{p_0}(B(R))} \le CC_\epsilon R^\epsilon K^{-2+\frac{4}{p_0}} \| f_\tau^K \|_{L^\infty(S)} = CC_\epsilon R^\epsilon K^{-2+\frac{4}{p_0}} \| f_\tau \|_{L^\infty(S)}.
\]

Now we prove \eqref{eqn:sumtau}.
Let $L=L_{\parallel e_1}$; when $L=L_{\parallel e_2}$ the argument below is similar. By translation we may assume that the center line of $L$ is $e_1$.
Let $f_L = \sum_{\tau: \bar\tau \subset L } f_\tau$. 
Taking a rescaling $(\xi_1,\xi_2) \rightarrow (\xi_1, K^{-1}\xi_2)$, we have
\[ 
Ef_L(x_1,x_2,x_3)  = K^{-1} [Ef_L^K](x_1,K^{-1}x_2, K^{-1}x_3)
\]
where $f_L^{K} = f_L(\xi_1,K^{-1}\xi_2)$. We can see that $f_L^K$ is supported in $[-1,1]^2$. By changing of variables, 
\[
\|  [Ef_L^K](\cdot,K^{-1}\cdot, K^{-1}\cdot) \|_{L^{p_0}(B(R))} = K^{2/p_0}\|  Ef_L^K  \|_{L^{p_0}(L^*)} 
\]
where $L^*$ is a tube of dimensions $R \times R/K \times R/K$. Thus, combining these, we have
\[ 
\| Ef_L \|_{L^{p_0}(B(R))} \le  K^{-1+\frac{2}{p_0}} \|  Ef_L^K  \|_{L^{p_0}(L^*)}.
\]
Cover $L^*$ with two balls of radius $\frac{3}{4}R$. Since $\| Ef_L\|_{L^{p_0}(L^* \cap B(\frac{3R}{4}))} \le \| Ef_L\|_{L^{p_0}(B(\frac{3R}{4}))}$, we can apply the induction hypothesis to each ball. So, we obtain
\[ 
\| Ef_L\|_{L^{p_0}(B(R))} \le CC_\epsilon R^\epsilon K^{-1+\frac{2}{p_0}} \| f_L^K \|_{L^\infty(S)} = CC_\epsilon R^\epsilon K^{-1+\frac{2}{p_0}} \| f_L \|_{L^\infty(S)}.
\]
\end{proof}

Let us set $\alpha = K^{-\epsilon}$.
After raising both sides in \eqref{eqn:tau} to the $p_0$th power, we sum these over $\tau$. Since the number of caps $\tau$ is $K^2$, we have
\[
K^{O(\epsilon)} \sum_{\tau} \int_{B(R)} | Ef_{\tau}(x)|^{p_0} \le C_\epsilon^{p_0} (C K^{-2p_0+6+O(\epsilon)})R^{ p_0 \epsilon} \|f \|_{L^\infty(S)}^{p_0}.
\]
Since \( p_0 > 3 \) and $\lim_{\epsilon \rightarrow 0} K(\epsilon) = \infty$, we can take a sufficiently small \( \epsilon > 0 \) such that \[ 
CK^{-2p_0+6+O(\epsilon)} \le (1/3)^{p_0}.
\] 
So, it gives
\begin{equation} \label{nbr1}
K^{O(\epsilon)} \sum_{\tau} \int_{B(R)} | Ef_{\tau}(x)|^{p_0} \le  3^{-p_0} C_\epsilon^{p_0} R^{p_0 \epsilon} \|f \|_{L^\infty(S)}^{p_0}.
\end{equation}

Similarly, we raise both sides in \eqref{eqn:sumtau} to the $p_0$th power, and sum these over $L$. Then, since the number of strips $L$ is $K$, we have
\[
K^{O(\epsilon)} \sum_{L} \int_{B(R)} \Big|\sum_{\tau: \bar\tau \subset L } Ef_{\tau}(x) \Big|^{p_0} \le C_\epsilon^{p_0} (C K^{-{p_0}+3+O(\epsilon)})R^{p_0 \epsilon } \|f \|_{L^\infty(S)}^{p_0}.
\] 
From \( {p_0} > 3 \) and $\lim_{\epsilon \rightarrow 0} K(\epsilon) = \infty$, it is possible to take a sufficiently small \( \epsilon > 0 \) so that \( CK^{-{p_0}+3+O(\epsilon)} \le (1/3)^{p_0} \). Then,
\begin{equation} \label{nbr2}
K^{O(\epsilon)} \sum_{L} \int_{B(R)} \Big|\sum_{\tau: \bar\tau \subset L } Ef_{\tau}(x) \Big|^{p_0} \le  3^{-p_0} C_\epsilon^{p_0} R^{p_0 \epsilon} \|f \|_{L^\infty(S)}^{p_0}.
\end{equation}

To show \eqref{main estimate}, by \eqref{eqn:star}, \eqref{nbr1} and \eqref{nbr2} it suffices to prove \begin{equation} 
\| \mathbf{B}_{K^{-\epsilon}} [E f] \|_{L^{p_0}(B(R))} \le 3^{-1}C_\epsilon R^{\epsilon} \|f\|_{L^\infty(S)}.
\label{Goal2}
\end{equation}
This immediately follows from the following.

\begin{thm}\label{main part}	
Let $R \gg 1$. For any $0< \epsilon \ll 1$, there exists  $\delta_2 \in (0,\epsilon)$ and $K=K(\epsilon)$ with $\lim_{\epsilon \rightarrow 0} K(\epsilon) = \infty$  such that
if for any $\omega \in S$,
\begin{equation} \label{AvC}
\oint_{B(\omega,R^{-1/2}) \cap S} |f|^2 \le 1,
\end{equation}
then  
\begin{equation} \label{eqn:Goal}
\int_{B(R)} |\mathbf{B}_{K^{-\epsilon}} [E f]|^{p_0} \le C_\epsilon R^{\epsilon} R^{\delta_2}\bigg( \int_S |f|^2 \bigg)^{3/2+\epsilon}.
\end{equation}
Here, $C_\epsilon$ is independent of $R$ and $f$.
\end{thm}
Indeed, the implication from Theorem \ref{main part} to \eqref{Goal2} can be proven as follows.
We may assume that $\|f\|_{L^\infty(S)} \le 1$ by normalization. 
Then, it is easy to see that $\oint_{B(\omega,R^{-1/2}) \cap S} |f|^2 \le \|f\|_{L^\infty(S)}^2$ for any $\omega \in S$.
From \eqref{AvC} it follows that $\int_S |f|^2 \lesssim \sum_{\Omega} \int_{\Omega} |f|^2 \lesssim 1$. Combining this with the above estimate we have
\[ 
\| \mathbf{B}_{K^{-\epsilon}} [E f] \|_{L^{p_0}(B(R))} \le CC_\epsilon^{1/p_0} R^{2\epsilon/p_0}. 
\]
Since $\epsilon>0$ is arbitrary, this gives $\eqref{Goal2}$.
Now, it remains to prove Theorem \ref{main part}. This will be done in the next section. 

\begin{rmk}
The broad function defined in this paper is different from that in \cite{guth2015restriction}. This new broad function guarantees that the bilinear operator in Lemma \ref{lem:Maindecomp} has a stronger separation condition than that in \cite{guth2015restriction}. 
\end{rmk}

\section{Proof of Theorem \ref{main part}}
This section is devoted to prove Theorem \ref{main part}. We first mention a \textit{polynomial partitioning} which is a technique recently applied to some problems in incidence geometry. 

For a function $f$, we define the zero set of $f$ by \( Z(f) = \{ x : f(x) =0 \}\). For a polynomial $P$, we say that a polynomial \( P \) is \textit{non-singular} if it satisfies \( \nabla P(x) \neq 0 \) for each point $x$ in \( Z(P) \). It is known that non-singular polynomials are dense in the vector space of polynomials on $\mathbb R^n$ of degree at most $M$. 
The following is a polynomial partitioning involving non-singular polynomials. 
%
\begin{thm}[Polynomial partitioning for non-singular polynomials, \cite{guth2015restriction}] \label{thm:polyPart}
Assume that a nonnegative function \( f \in L^1(\mathbb R^n) \) is given. Then for each \( M=1,2,\cdots, \) there exists a non-zero polynomial \( P \) of degree at most \( M \) such that 
\[ 
\mathbb R^n \setminus Z(P) = \biguplus_{i=1}^{O(M^n)} O_i \] and all \( \int_{O_i} f \) are comparable. Moreover, the polynomial $P$ is a product of non-singular polynomials. 
\end{thm}

Now we prove Theorem \ref{main part}. To begin with, let us set 
\begin{equation} \label{deltas}
\delta = \epsilon^2 ,~  \delta_{1} =\epsilon^4 ~ \text{ and }~ \delta_{2} = \epsilon^6.
\end{equation}
Then we have the relation  
\(\ 
\epsilon \gg \delta \gg \delta_{1} \gg \delta_2.
\) 
(This relation plays a crucial role to close the induction below.)

For fixed $R \gg 1$, we analyze $\mathbf{B}_{\alpha} [E f]$. First, we apply Theorem \ref{thm:polyPart} with 
\begin{equation} \label{PolDeg}
M = R^{\delta_1} 
\end{equation}
to \( \chi_{B(R)} |\mathbf{B}_{\alpha} [E f]|^{p_0} \). Then there exists a non-zero polynomial $P$ of degree at most $M$ such that \[ \mathbb R^n \setminus Z(P) = \biguplus_i^{O(M^3)} O_i \] and for each $i$, it satisfies
\begin{equation} \label{eqn:prtcomp}
\int_{B(R) \cap O_i} |\mathbf{B}_\alpha [Ef]|^{p_0} \sim M^{-3} \int_{B(R)} |\mathbf{B}_\alpha [Ef]|^{p_0}. 
\end{equation}

Let us define the \textit{wall} $W$ by
the \( R^{1/2 + \delta} \)-neighborhood of \( Z(P) \)
and the \textit{cell} \( O_i'\) by \( O_i  \setminus W\).
Then we have
\begin{equation}\label{broad function estimate}
\int_{B(R)} |\mathbf{B}_\alpha [Ef]|^{p_0} = \sum_i \int_{B(R) \cap O_i'} |\mathbf{B}_\alpha [Ef]|^{p_0}  + \int_{B(R) \cap W} |\mathbf{B}_\alpha [Ef]|^{p_0}.
\end{equation}

To estimate the above we will use two kinds of induction. The first one is an induction on the scale $R$. We assume that for any $1 \le r \le R/2$, Theorem \ref{main part} is true. If $R=1$ then it is easy to see that the estimate \eqref{eqn:Goal} holds.
The other one is an induction on $\|f\|_{L^2(S)}$.  We assume that for all $g$ with $\|g\|_{L^2(S)} \le \frac{1}{2} \|f\|_{L^2(S)}$, Theorem \ref{main part} is true. 
If $\|g \|_{L^2(S)} \le R^{-1000}$ then we can easily obtain \eqref{eqn:Goal}.

\subsection{Cell estimate} 
We consider the contribution of the summation part of the right side of \eqref{broad function estimate}. To deal with this part we will use the second induction.
Suppose that this summation part  dominates the other term. Then, 
\begin{equation} \label{cellDorm}
\int_{B(R)} |\mathbf{B}_\alpha [Ef]|^{p_0}  \lesssim \sum_i \int_{B(R) \cap O_i'} |\mathbf{B}_\alpha [Ef]|^{p_0}.
\end{equation}

\begin{lem}
Assume \eqref{eqn:prtcomp} and \eqref{cellDorm}.
Then there exists a subcollection $\mathcal I$ with cardinality $O(M^3)$ such that for all $i \in \mathcal I$, 
\begin{equation} \label{eqn:cell1}
\int_{B(R) \cap O_i'} |\mathbf{B}_\alpha [Ef]|^{p_0} \sim M^{-3} \int_{B(R)} |\mathbf{B}_\alpha [Ef]|^{p_0}.
\end{equation}
\end{lem}
\begin{proof}
For convenience, let $X_i:= \int_{B(R) \cap O_i'} |\mathbf{B}_\alpha [Ef]|^{p_0}$ and $A:=M^{-3}\int_{B(R)} |\mathbf{B}_\alpha [Ef]|^{p_0}$, and  let $N$ be the number of cells $O_i$. Then from \eqref{eqn:prtcomp} we see that there exists a constant $C_1 \ge 1$ such that for each $i$,
\begin{equation} \label{s1}
X_i \le C_1A,
\end{equation}
and from \eqref{cellDorm} we see that there exists a constant $C_2 \ge 1$ such that 
\begin{equation} \label{s2} 
N A \le C_2 \sum_{i=1}^N X_i.
\end{equation}

Let $c_*$ be a small positive number which will be chosen later. 
Suppose that there are  $\lambda N$ cells, $\lambda \in [0,1]$, such that $X_i \ge c_*A$ for all $i$. Then it suffices to show $\lambda \sim 1$. In \eqref{s2} we decompose $\sum X_i$ into two parts as follows:
\[
C_2^{-1} NA \le \sum_{X_i \ge c_*A}^{\lambda N} X_i + \sum_{X_i < c_*A}^{(1-\lambda) N} X_i. 
\]
By \eqref{s1}, it is bounded by
\begin{align*}
&\le \lambda N C_1A + \sum_{X_i < c_*A}^{(1-\lambda) N} X_i \\
&\le \lambda N C_1A + (1-\lambda)N c_*A.
\end{align*}
By dividing the above by $NA$, we have
\(
C_2^{-1} \le C_1\lambda + (1-\lambda)c_*.
\)
By rearranging we have $\lambda \ge \frac{C_2^{-1} - c_*}{C_1-c_*}$ provided $0< c_* \le \frac{1}{2C_2}$. It means $\lambda \sim 1$. 
\end{proof}

We rewrite \eqref{eqn:cell1} as
\begin{equation}\label{eqn:cell3}
\int_{B(R)} |\mathbf{B}_\alpha [Ef]|^{p_0} \sim M^3 \int_{B(R) \cap O_i'} |\mathbf{B}_\alpha [Ef]|^{p_0}. 
\end{equation}
We will apply  the second induction hypothesis to the above. For this we need several lemmas for restricting the wavepackets $f_T$ to those with $T$ passing through $O_i'$. 
We decompose $f$ into the wave packets on $B(R)$. 
By (3) of Proposition \ref{prop:wavepack} we may set 
\begin{equation} \label{WP}
f = \sum_{T \in \mathbb T} f_T.
\end{equation} 
Then, $f_\tau$ can be written as
\[
f_\tau = \sum_{T \in \mathbb T : \omega(T) \in \tau}f_T.
\]

\noindent For each $i$ and $\tau$, let us define $f_{\tau,i}$ and $f_i$ by
\[
f_{\tau,i} = \sum_{T \in \mathbb T_i : \omega(T) \in \tau} f_{T} \quad \text{and} \quad
f_i = \sum_{\tau} f_{\tau,i}
\]
respectively. 
We will consider the wave packets $f_T$ with $T \cap O_i' \neq \emptyset$.
Let $\mathbb T_i(\Omega)$ be the subcollection defined by
\[
\mathbb T_i(\Omega) = \{ T \in \mathbb T(\Omega) : T \cap O_i' \neq \emptyset \},
\] 
and let $\mathbb T_i = \bigcup_\Omega \mathbb T_i(\Omega)$.

\begin{lem}\label{broad lemma}
For $x \in O_i'$,
\begin{equation} \label{eqn:cell2}
|\mathbf{B}_\alpha [Ef](x)|
\lesssim |\mathbf{B}_{4\alpha} [Ef_i](x)| +  R^{-900}\sum_{\tau}\|f_{\tau}\|_{L^2(S)}.
\end{equation}
\end{lem}

\begin{proof}
Using \eqref{WP} we decompose $Ef$ as
\[
Ef =\sum_{T \in \mathbb T} Ef_T = \sum_{T \in \mathbb T_i} Ef_T + \sum_{T \in \mathbb T \setminus \mathbb T_i} Ef_T.
\]
From (2) of Proposition \ref{prop:wavepack}, it follows that for $x \in O_i'$,
\begin{equation}\label{cell part} 
Ef(x) = Ef_{i}(x) + O\big(R^{-990}\sum_{\tau}\|f_{\tau}\|_{L^2(S)} \big).
\end{equation}

Now it suffices to show that if $x \in O_i'$ is an $\alpha$-broad point of $Ef$ then $x$ is also a $4\alpha$-broad point of $Ef_i$.
We may assume $|Ef_i(x)| \ge R^{-900}\sum_{\tau}\|f_{\tau}\|_{L^2(S)}$; otherwise, from \eqref{cell part} we have 
\[ 
|\mathbf{B}_\alpha [Ef](x)| \le  |Ef(x)| \le |Ef_{i}(x)| + O \big(R^{-990}\sum_{\tau}\|f_{\tau}\|_{L^2(S)} \big) \lesssim R^{-900}\sum_{\tau}\|f_{\tau}\|_{L^2(S)},
\]
which satisfies \eqref{eqn:cell2}.
Since $x \in O_i'$ is an $\alpha$-broad point of $Ef$, we have that for any cap $\tau$,
\begin{align*}
|Ef_{\tau,i}(x)| &\le| Ef_{\tau}(x)| + O(R^{-990}\|f_{\tau}\|_{L^2(S)}) \\ &\le
\alpha |Ef(x)| + O(R^{-990}\|f_{\tau}\|_{L^2(S)}).
\end{align*}
From \eqref{cell part} it is bounded by $\alpha |Ef_i(x)| + O(R^{-990} \sum_{\tau}\|f_{\tau}\|_{L^2(S)})$. So, for large $R$, it implies that for any $\tau$,
\[ 
|Ef_{\tau,i}(x)| \le
 2\alpha |Ef_i(x)|.
 \]
Similarly, for any $L= L_{\parallel e_1}$ or $L_{\parallel e_2}$, we have that for any $x \in O_i'$,
\begin{align*}
\Big|\sum_{\tau: \bar\tau \subset L } Ef_{\tau,i}(x) \Big| &\le \Big| \sum_{\tau: \bar\tau \subset L } Ef_{\tau}(x) \Big| + O(R^{-990}\sum_{\tau: \bar\tau \subset L } \|f_{\tau}\|_{L^2(S)}) \\ 
&\le \alpha |Ef(x)| + O(R^{-990}\sum_{\tau: \bar\tau \subset L } \|f_{\tau}\|_{L^2(S)}) \\ 
&\le \alpha |Ef_i(x)| + O(R^{-990}\sum_{\tau}\|f_{\tau}\|_{L^2(S)}) \\
&\le 2\alpha |Ef_i(x)|.
\end{align*}
From these it follows that for any $\alpha$-broad point $x \in O_i'$,
\[ 
\max_{\tau} |Ef_{\tau,i}(x)| + \max_{L=L_{\parallel e_1} \text{ or } L_{\parallel e_2}} \Big|\sum_{\tau: \bar\tau \subset L } Ef_{\tau,i}(x) \Big| \le 4\alpha |Ef_i(x)|.
\] 
\end{proof}

We raise both sides in \eqref{eqn:cell2} to the $p_0$th power and integrate it over $B(R) \cap O_i'$. Then
\begin{align} 
\int_{B(R) \cap O_i'} |\mathbf{B}_\alpha [Ef]|^{p_0} 
&\lesssim \int_{B(R) \cap O_i'} |\mathbf{B}_{4\alpha} [Ef_i]|^{p_0} + K^2 R^{-2000}\sum_{\tau}\|f_{\tau}\|_{L^2(S)}^{p_0} \nonumber \\
\label{Gfi}
&\lesssim \int_{B(R)} |\mathbf{B}_{4\alpha} [Ef_i]|^{p_0} + K^2 R^{-2000}\sum_{\tau}\|f_{\tau}\|_{L^2(S)}^{p_0}.
\end{align}

From (6) in Proposition \ref{prop:wavepack}, we have
\begin{align*}
\oint_{B(\omega, R^{-1/2}) \cap S} |f_{\tau,i}|^2 \lesssim  \oint_{10B(\omega, R^{-1/2}) \cap S} |f_{\tau}|^2 \lesssim 1.
\end{align*}
So, to apply the second induction hypothesis to \eqref{Gfi}, it remains to show 
\begin{equation} \label{chconIn}
\|f_{i}\|_{L^2(S)} \le \frac{1}{2} \|f\|_{L^2(S)}.
\end{equation}

We first prove the following lemma by using the geometric fact that 
if $P$ is a non-zero polynomial of degree $M$ then the algebraic surface $Z(P)$ intersects a line in at most $M$ points. 
 
\begin{lem}\label{sake estimate}
\begin{equation} \label{degCount}
\sum_{i} \int |f_{\tau,i}|^2 \lesssim M \int |f_{\tau}|^2 + R^{-900}  \|f_\tau\|_{L^2(S)}^2.
\end{equation}
\end{lem}
\begin{proof}
From (1) of Proposition \ref{prop:wavepack} we have that for each $i$,
\begin{align*}
\int |f_{\tau,i}|^2 &\lesssim \int \Big| \sum_{\Omega: \Omega \cap \tau \neq \emptyset} \sum_{T \in \mathbb T_i(\Omega) } f_{T} \Big|^2 \\
&\lesssim  \sum_{\Omega: \Omega \cap \tau \neq \emptyset}  \int \Big| \sum_{T \in \mathbb T_i(\Omega) }  f_{T} \Big|^2. 
\end{align*}
From (4) of Proposition \ref{prop:wavepack} it follows that for each $i$,
\[
\int |f_{\tau,i}|^2 \lesssim  \sum_{\Omega: \Omega \cap \tau \neq \emptyset}  \sum_{T \in \mathbb T_i(\Omega)} \int |  f_{T}|^2 + R^{-950}  \|f_\tau\|_{L^2(S)}^2.
\]
By summing over $i$,
\begin{align*}
\sum_i \int |f_{\tau,i}|^2 &\lesssim  \sum_{\Omega: \Omega \cap \tau \neq \emptyset} \sum_i  \sum_{T \in \mathbb T_i(\Omega)} \int |  f_{T}|^2 + R^{-900}  \|f_\tau\|_{L^2(S)}^2\\
&\lesssim \sum_{\Omega: \Omega \cap \tau \neq \emptyset} \sum_{T \in \mathbb T(\Omega)} \sum_{i: O_i' \cap T \neq \emptyset } \int |  f_{T}|^2 + R^{-900} \|f_\tau\|_{L^2(S)}^2.
\end{align*}
We observe that each tube \( T \in \mathbb T \) intersects \( O_i \) at most $(M+1)$ times because a line can cross \( Z(P) \) at most $M$ times. It makes 
\[
\sum_{i} \int |f_{\tau,i}|^2
\lesssim M \sum_{\Omega: \Omega \cap \tau \neq \emptyset} \sum_{T \in \mathbb T(\Omega)} \int |  f_{T}|^2 + R^{-900}  \|f_\tau\|_{L^2(S)}^2.
\]
From (1) and (4) of Proposition \ref{prop:wavepack}, we can finally obtain \eqref{degCount}.
\end{proof}

We sum \eqref{degCount} over $\tau$, and then we use the pigeonhole principle to select an $i_0 \in \mathcal I$ such that 
\begin{equation} \label{cttau}
\sum_{\tau} \int |f_{\tau,i_0}|^2 \lesssim M^{-2} \sum_{\tau} \int |f_{\tau}|^2 + M^{-3}R^{-900} \sum_{\tau} \|f_\tau\|_{L^2(S)}^2.
\end{equation}
Since $S$ is covered by caps $\tau$, it means that $\|f_{i_0}\|_{L^2(S)}^2 \le (CM^{-2} + M^{-3}R^{-900}) \|f\|_{L^2(S)}^2$.
Thus, by \eqref{PolDeg} we have \eqref{chconIn} for sufficiently large $R$. Now we apply the second induction hypothesis to \eqref{Gfi} with $i=i_0$.
Then it gives that
\begin{equation*} 
\int_{B(R) \cap O_{i_0}'} |\mathbf{B}_{\alpha} [Ef_{i_0}]|^{p_0} \lesssim C_\epsilon R^{\epsilon}R^{\delta_2} \Big( \sum_{\tau} \int_S |f_{\tau,i_0}|^2 \Big)^{3/2+\epsilon}
 + K^2 R^{-2000}\sum_{\tau}\|f_{\tau}\|_{L^2(S)}^{p_0}.
\end{equation*}
By substituting this in \eqref{eqn:cell3}, one has 
\begin{equation*}
\int_{B(R)} |\mathbf{B}_\alpha [Ef]|^{p_0}
\lesssim C_\epsilon M^3 
R^{\epsilon}R^{\delta_2} \Big( \sum_{\tau} \int_S |f_{\tau,i_0}|^2 \Big)^{3/2+\epsilon} +  K^2R^{-2000}\sum_{\tau}\|f_{\tau}\|_{L^2(S)}^{p_0}.
\end{equation*}
By \eqref{cttau}, it is bounded by
\begin{align*}
&\lesssim C_\epsilon ( M^{-2\epsilon} R^{\delta_2})
R^{\epsilon}  \bigg( \sum_{\tau} \int_S |f_{\tau}|^2 \bigg)^{3/2+\epsilon}
+ K^2R^{-1000}\sum_{\tau}\|f_{\tau}\|_{L^2(S)}^{p_0} \\
&\lesssim C_\epsilon R^\epsilon ( M^{-2\epsilon} R^{\delta_2} + K^2 R^{-1000}) \bigg( \sum_{\tau} \int_S |f_{\tau}|^2 \bigg)^{3/2+\epsilon},
\end{align*}
where we used the estimate $\|f\|_{L^2(S)} \lesssim 1$,
(which follows from the condition \eqref{AvC}; $\|f\|_{L^2(S)}^2 \lesssim \sum_{\Omega} \int |f_{\Omega}|^2 \lesssim 1$).

From \eqref{PolDeg} and \eqref{deltas}, one has $M^{-2\epsilon} R^{\delta_2} = R^{-2\epsilon^5+\epsilon^6}$. 
Since the exponent of $R$ is negative, we have $CR^{-2\epsilon^5+\epsilon^6} + CK^2R^{-1000}\le 1/2$ for sufficiently large $R$. Thus we obtain
\[ 
\int_{B(R)} |\mathbf{B}_\alpha [Ef]|^{p_0}
\le 2^{-1} C_\epsilon R^\epsilon  \bigg( \sum_{\tau} \int_S |f_{\tau}|^2 \bigg)^{3/2+\epsilon}.
\]

\subsection{Wall estimate} 
Now we suppose that the integral $\int_{B(R) \cap W} |\mathbf{B}_\alpha[Ef]|^{p_0}$ dominates the other term in \eqref{broad function estimate}. Then it suffices to prove 
\begin{equation} \label{WEST}
\int_{B(R) \cap W} |\mathbf{B}_\alpha[Ef]|^{p_0} \le 2^{-1} C_\epsilon R^\epsilon R^{\delta_2} \bigg( \sum_{\tau} \int_S |f_{\tau}|^2 \bigg)^{3/2+\epsilon}.
\end{equation}

We split the wave packets $f_T$ into transverse ones  and tangent ones to the wall $W$.  
We first cover $B(R)$ with \( O(R^{3\delta}) \) balls $B_j$ of radius $R^{1-\delta}$. (Later we will use an inductive argument to each $B_j$ to estimate the transversal part.) 

We define the collection \( \mathbb T_{j}^{\flat} \) of \textit{tangential tubes}  to be the collection of all tubes $T \in \mathbb T$ such that \( T \cap W \cap B_j \neq \emptyset \) and
if $z$ is any non-singular point of \( Z(P) \) in \( 2B_j \cap 10T \), then
\[
\angle(v(T), T_zZ) \le R^{-1/2 + 2\delta},
\]
where $T_zZ$ is the tangent plane of $Z(P)$ at a point $z$. 
We also define the collection $T \in \mathbb T$ of \textit{transversal tubes}
\( \mathbb T_{j}^{\sharp}\) to be the collection of all tubes such that \( T \cap W \cap B_j \neq \emptyset \) and there exists a non-singular point $z$ of \( Z(P) \) in \( 2B_j \cap 10T \) so that
\[
\angle(v(T), T_zZ) > R^{-1/2 + 2\delta}.
\]
If $I$ is a subcollection of the caps $\tau$, we define $f_I$ by
\[ 
f_{I} :=  \sum_{\tau \in I} f_{\tau},
\]
and set  
\[
f_{\tau,j}^{\sharp} := \sum_{T \in \mathbb T_j^{\sharp} : \omega(T) \in \tau} f_{T},\quad  f_j^{\sharp} := \sum_{\tau} f_{\tau,j}^{\sharp}  \quad\text{and}\quad
f_{I,j}^{\sharp} :=  \sum_{\tau \in I} f_{\tau,j}^\sharp, 
\]
and similarly define $f_{\tau,j}^\flat$, $f_j^{\flat}$  and $f_{I,j}^\flat$.

We will consider a bilinear form of $Ef$ under a certain separation condition. 
For $A, B \subset \mathbb R^2$, let $\dist_{\xi_i}(A,B) :=\dist(\mathrm{proj}_{\xi_i}A,\mathrm{proj}_{\xi_i}B)$, where $\mathrm{proj}_{\xi_i}$ is a projection to $\xi_i$-axis.
We define the bilinear operator $\mathrm{Bil}(Ef)$ as
\[
\mathrm{Bil}(Ef) := \sum_{\substack{(\tau_1,\tau_2) : \dist_{\xi_1}(\bar\tau_1,\bar\tau_2) \ge \frac{1}{2}K^{-1},\\ \dist_{\xi_2}(\bar\tau_1,\bar\tau_2) \ge \frac{1}{2}K^{-1}}} |Ef_{\tau_1}|^{1/2} |Ef_{\tau_2}|^{1/2}.
\]
By using the definition of broad point we can decompose $\mathbf{B}_\alpha [Ef]$ as follows. 

\begin{lem} \label{lem:Maindecomp}
Let $\alpha = K^{-\epsilon}$. Then, for any $x \in B_j \cap W$, 
\begin{equation} \label{BilDec}
|\mathbf{B}_\alpha [Ef](x)| \lesssim \sum_I | \mathbf B_{C\alpha} [Ef_{I,j}^{\sharp}](x)| +K^{100} \mathrm{Bil}(Ef_{j}^{\flat})(x) + R^{-900}\sum_\tau \|f_\tau\|_{L^2(S)},
\end{equation}
where $I$ runs over all subcollections consisting of caps $\tau$.
\end{lem}
\begin{proof}
Suppose that $x \in B_j \cap W$ is an $\alpha$-broad point of $Ef$.
We assume that 
\begin{equation} \label{asumTrivial}
|Ef(x)| \ge CR^{-900} \sum_\tau \|f_\tau\|_{L^2(S)};
\end{equation} 
otherwise, it trivially  gives \eqref{BilDec}.

Let \( I \) be the collection of caps $\tau$ satisfying
\begin{equation} \label{eqn:Ide}
|Ef_{\tau,j}^{\flat}(x)| \le K^{-100}|Ef(x)|.
\end{equation}
Consider the complement $I^c$. 
We will say that caps $\tau_1$ and $\tau_2$ are \textit{strong-separated} if  $\dist_{\xi_1}(\bar\tau_1,\bar\tau_2) \ge \frac{1}{2}K^{-1}$ and $\dist_{\xi_2}(\bar\tau_1,\bar\tau_2) \ge \frac{1}{2}K^{-1}$.
If $I^c$ has two strong-separated caps $\tau_1$, $\tau_2$, then one has 
\begin{equation} \label{BoBil}
|Ef(x)| \le K^{100} | Ef^\flat_{\tau_1,j}(x) |^{1/2} | Ef^\flat_{\tau_2,j}(x) |^{1/2},
\end{equation}
since  
$|Ef(x)| < K^{100} | Ef^\flat_{\tau,j}(x) |$ for all $\tau \in I^c$.

\

We now suppose that $I^c$ does not have any strong-separated pair of caps.
We claim that there exists a strip $\tilde L$ of width $\le 8K^{-1}$ which is  parallel to $e_1$ or $e_2$ and contains all $\bar \tau$ for $\tau \in I^c$. 

Let us prove the claim. By abusing notations, we identify a cap $\tau$ with the projected cap $\bar\tau$. Fix a cap $\tau_0 \in I^c$. 
Let $A_j = \{ \tau \in I^c: \dist_{\xi_j}(\tau_0,\tau) < \frac{1}{2}K^{-1} \}$ for $j=1,2$, and let
$A_0 =\{ \tau \in I^c : \dist_{\xi_1}(\tau_0,\tau) < \frac{3}{2}K^{-1} \text{ and } \dist_{\xi_2}(\tau_0,\tau) < \frac{3}{2}K^{-1} \}$. Then since every $\tau \in I^c$ is not strong-separated to $\tau_0$, one has $I^c = A_1 \cup A_2$. 
Observe that if $\tau_1 \in A_1 \setminus A_0$ and $\tau_2 \in A_2 \setminus A_0$ then  $\tau_1$ and $\tau_2$ are strong-separated. Thus, one has that $A_1 \setminus A_0 = \emptyset$ or $A_2 \setminus A_0 = \emptyset$ by the supposition. If $A_1 \setminus A_0$ is nonempty, we can take a strip of width $8K^{-1}$ and of being parallel to $e_2$ which contains both $A_0$ and $A_1$. For a case of $A_2 \setminus A_0 \neq \emptyset$ we can take a similar strip of being parallel to $e_1$. Therefore, we have the claim.

Let $\acute I^c$ be the collection of caps $\tau$ with $\bar\tau \cap \tilde L \neq \emptyset$, and $\acute I = I \setminus \acute I^c$.
Since $x$ is an $\alpha$-broad point of $Ef$, we have
\begin{align*}
|Ef(x)| 
&\le \Big|\sum_{\tau \in \acute I} Ef_{\tau}(x) \Big| + \Big|\sum_{\tau \in \acute I^c} Ef_{\tau}(x) \Big| \\
&\le \Big|\sum_{\tau \in \acute I} Ef_{\tau}(x) \Big| +  16\max_{L=L_{\parallel e_1} \text{ or } L_{\parallel e_2}} \bigg| \sum_{\tau:\bar\tau \subset L} Ef_{\tau}(x) \bigg| \\
&\le |Ef_{\acute I}(x)| + 16 \alpha |Ef(x)|.
\end{align*}
If $K$ is large enough, then one has $0< 16\alpha < 1/2$. So, by rearranging this we get
\begin{equation} \label{I^c}
|Ef(x)| \lesssim | Ef_{\acute I}(x) |.
\end{equation}

Now we decompose $f_{\acute I}$ into 
\(
f_{\acute I}  = \sum_{\tau \in \acute I} \sum_{T: \omega(T) \in \tau} f_T. 
\)
By (2) in Proposition \ref{prop:wavepack} we can ignore $Ef_T$ with $T \cap (B_j \cap W) = \emptyset$, and so we have
\[
\Big| \sum_{T :\omega(T) \in \tau,~ T \cap  B_j \cap W= \emptyset}Ef_{T}(x) \Big| \lesssim R^{-990}\|f_\tau\|_{L^2(S)}.
\]
Each tube $T$ intersecting $B_j \cap W$ is contained in either $\mathbb T_{j}^{\flat}$ or $\mathbb T_{j}^{\sharp}$. So, for each $\tau \in I'$, 
\begin{equation} \label{Eftau}
Ef_\tau(x) = Ef_{\tau,j}^{\sharp}(x) + Ef_{\tau,j}^{\flat}(x) + O(R^{-990}\|f_\tau\|_{L^2(S)}).
\end{equation}
By summing over \( \tau \in \acute I \), 
\begin{equation*} 
|Ef_{\acute I}(x)| \le |Ef_{\acute I,j}^{\sharp}(x)| + \sum_{\tau \in \acute I} |Ef_{\tau,j}^{\flat}(x)| + O(R^{-990}\sum_\tau\|f_\tau\|_{L^2(S)}).
\end{equation*}
From \eqref{eqn:Ide} and $I' \subset I$ it follows that 
\begin{equation} \label{IconR}
\sum_{\tau \in \acute I} |Ef_{\tau,j}^{\flat}(x)| \le K^{-98}|Ef(x)|.
\end{equation}
Inserting this into the previous inequality, we obtain
\begin{equation*} 
|Ef_{\acute I}(x)| \le |Ef_{\acute I,j}^{\sharp}(x)| + K^{-98}|Ef(x)| + O(R^{-990}\sum_\tau\|f_\tau\|_{L^2(S)}),
\end{equation*}
and by \eqref{asumTrivial},
\begin{equation*} 
|Ef_{\acute I}(x)| \le |Ef_{\acute I,j}^{\sharp}(x)| + K^{-98}|Ef(x)| +C R^{-90}|Ef(x)|.
\end{equation*}
We combine this with \eqref{I^c} and rearrange it. Then it follows that
\begin{equation} \label{bdEdsh} 
|Ef(x)| \lesssim  |Ef_{\acute I,j}^{\sharp}(x)|.
\end{equation}

To prove \(
|\mathbf B_\alpha [Ef](x)| \lesssim  |\mathbf B_{C\alpha}[Ef_{\acute I,j}^{\sharp}](x)|,
\) 
it remains to show that if $x \in B_j \cap W$ is an $\alpha$-broad point of $Ef$ then $x$ is also a $C\alpha$-broad point of $Ef_{\acute I,j}^{\sharp}$.
Let $\tau \in I'$ be given. By \eqref{Eftau} we have 
\[
| Ef_{\tau,j}^{\sharp}(x)| \le |Ef_{\tau,j}(x)| +  |Ef_{\tau,j}^{\flat}(x)| + O(R^{-990} \|f_\tau\|_{L^2(S)}).
\]
From (2) in Proposition \ref{prop:wavepack}, we have that \( |Ef_{\tau,j}(x)| \le |Ef_{\tau}(x)| + O(R^{-990} \sum_\tau \|f_\tau \|_{L^2(S)}) \) for $x \in B_j \cap W$.  So it follows that
\[
| Ef_{\tau,j}^{\sharp}(x)| \le |Ef_{\tau}(x)| +  |Ef_{\tau,j}^{\flat}(x)| + O(R^{-990} \|f_\tau\|_{L^2(S)}).
\]
Since $x$ is an $\alpha$-broad point of $Ef$, we have
\[
| Ef_{\tau,j}^{\sharp}(x)| \le \alpha |Ef(x)| +  |Ef_{\tau,j}^{\flat}(x)| + O(R^{-990} \|f_\tau\|_{L^2(S)}).
\]
From \eqref{eqn:Ide}, \eqref{asumTrivial} and $\alpha = K^{-\epsilon}$, we have
\[ 
|Ef_{\tau,j}^{\flat}(x)| + O(R^{-990} \|f_\tau\|_{L^2(S)})\le (K^{-100} + CR^{-90})|Ef(x)| \le \alpha |Ef(x)|
\]
for large $R$.
By substituting this in the previous one, we obtain
\begin{equation} \label{1broad}
| Ef_{\tau,j}^{\sharp}(x)| \le 2\alpha |Ef(x)|.
\end{equation}

Now, let $\varLambda = L_{\parallel e_1}$ or $L_{\parallel e_2}$.
By \eqref{Eftau},
\[
\Big| \sum_{\tau \in \acute I: \bar\tau \subset \varLambda} Ef_{\tau,j}^{\sharp}(x) \Big| \le \Big| \sum_{\tau \in \acute I: \bar\tau \subset \varLambda} Ef_{\tau,j}(x) \Big| + \sum_{\tau \in \acute I: \bar\tau \subset \varLambda} |Ef_{\tau,j}^{\flat}(x)| + O(R^{-990}\sum_{\tau}\|f_\tau\|_{L^2(S)}).
\]
If $\varLambda$ is parallel to $\tilde L$ then one has $|\sum_{\tau \in \acute I: \bar\tau \subset \varLambda} Ef_\tau(x)| \le |\sum_{\tau: \bar\tau \subset \varLambda} Ef_\tau(x)|$. If $\varLambda$ is perpendicular to $\tilde L$ then $|\sum_{\tau \in \acute I: \bar\tau \subset \varLambda} Ef_\tau(x)| \le |\sum_{\tau: \bar\tau \subset \varLambda} Ef_\tau(x)| + 16\max_{\tau}|Ef_\tau(x)|$. 
Thus, it gives
\begin{equation} \label{checkBr}
\begin{split}
\Big|\sum_{\tau \in \acute I: \bar\tau \subset \varLambda} Ef_{\tau,j}^{\sharp}(x) \Big| 
&\le  \Big|\sum_{\tau: \bar\tau \subset \varLambda} Ef_\tau(x) \Big| + 16\max_{\tau} |Ef_\tau(x)| \\
&\qquad +  \sum_{\tau \in \acute I: \bar\tau \subset \varLambda} |Ef_{\tau,j}^{\flat}(x)| + O(R^{-990}\sum_{\tau}\|f_\tau\|_{L^2(S)}) .
\end{split}
\end{equation}
Since $x$ is an $\alpha$-broad point of $Ef$, one has
\[ 
\Big|\sum_{\tau: \bar\tau \subset \varLambda} Ef_\tau(x) \Big| + 16\max_{\tau} |Ef_\tau(x)| \le 16 \alpha|Ef(x)|. 
\]
By \eqref{IconR} and \eqref{asumTrivial} we have 
\[ 
\sum_{\tau \in \acute I: \bar\tau \subset \varLambda} |Ef_{\tau,j}^{\flat}(x)| + O(R^{-990}\sum_{\tau}\|f_\tau\|_{L^2(S)})\le (K^{-98} + CR^{-90})|Ef(x)| \le \alpha |Ef(x)|
\]
for large $R$.
Inserting these two estimates into \eqref{checkBr}, it follows that
\[
\Big| \sum_{\tau \in \acute I: \bar\tau \subset \varLambda} Ef_{\tau,j}^{\sharp}(x) \Big| \le 20\alpha |Ef(x)|.
\]
To the sum of \eqref{1broad} and the above estimate, we apply \eqref{bdEdsh}. Then,
\[
|Ef^{\sharp}_{\tau,j}(x)| + \Big| \sum_{\tau \in \acute I: \bar\tau \subset \varLambda} Ef_{\tau,j}^{\sharp}(x) \Big| \le C\alpha |Ef_{\acute I, j}^{\sharp}(x)|.
\]
Thus, $x$ is a $C\alpha$-broad point of $Ef_{\acute I,j}^{\sharp}$ and so we have
\[
|\mathbf B_\alpha [Ef](x)| \lesssim  |\mathbf B_{C\alpha}[Ef_{\acute I,j}^{\sharp}](x)|.
\] 
By combining this and \eqref{BoBil} we have 
\[
|\mathbf B_\alpha [Ef](x)| \lesssim  |\mathbf B_{C\alpha}[Ef_{\acute I,j}^{\sharp}](x)| +  K^{100} | Ef^\flat_{\tau_1,j}(x) |^{1/2} | Ef^\flat_{\tau_2,j}(x) |^{1/2} + R^{-900}\sum_\tau \|f_\tau\|_{L^2(S)}.
\]
In this estimate, two strong-separated caps $\tau_1$, $\tau_2$ and $I'$ depend on $x \in B_j \cap W$.
To have independency we replace $|\mathbf B_{C\alpha}[Ef_{\acute I,j}^{\sharp}](x)| +  K^{100} | Ef^\flat_{\tau_1,j}(x) |^{1/2} | Ef^\flat_{\tau_2,j}(x) |^{1/2}$ with $ \sum_I | \mathbf B_{C\alpha} [Ef_{I,j}^{\sharp}](x)| +K^{100} \mathrm{Bil}(Ef_{j}^{\flat})(x)$. Then we obtain \eqref{BilDec}. 
\end{proof}

From Lemma \ref{lem:Maindecomp} it follows that
\begin{multline*}
\int_{B(R) \cap W} |\mathbf{B}_\alpha [Ef]|^{p_0} \\
\le C_\epsilon \Big( \sum_{j,I} \int_{B_j \cap W} | \mathbf{B}_{C\alpha} [Ef_{I,j}^{\sharp}]|^{p_0} +  K^{400} \sum_j\int_{B_j \cap W} \mathrm{Bil}(Ef_{j}^{\flat})^{p_0} + R^{-850} \sum_\tau \|f_{\tau}\|_{L^2(S)}^{p_0} \Big).
\end{multline*}

Now we will consider the transversal part $\sum_{j,I} \int_{B_j \cap W} | \mathbf{B}_{C\alpha} [Ef_{I,j}^{\sharp}]|^{p_0}$ and the tangential part $\sum_j\int_{B_j \cap W} \mathrm{Bil}(Ef_{j}^{\flat})^{p_0}$, respectively. (The last error term $R^{-850} \sum_\tau \|f_{\tau}\|_{L^2(S)}^{p_0}$ is trivially bounded by $R^{\epsilon}\big( \sum_{\tau} \int_S |f_{\tau}|^2 \big)^{3/2+\epsilon}$ for a sufficiently large $R$; for instance, we can use an estimate $\|f\|_{L^2(S)} \lesssim 1$ which follows from \eqref{AvC}.)
For the transversal part we will utilize the induction on scale $R$, and for the tangential part we will directly estimate it by using the bilinear method in \cites{lee2006bilinear, vargas2005restriction}.

\subsubsection{Estimate for the transversal part} 
We claim 
\begin{equation} \label{EstTrans}
\sum_{j,I} \int_{B_j} | \mathbf{B}_{C\alpha} [Ef_{I,j}^{\sharp}]|^{p_0}
\le C_\epsilon R^\epsilon  R^{\delta_2} \bigg( \sum_\tau \int_S |f_{\tau}|^2 \bigg)^{3/2+\epsilon}.                                                              \end{equation}
To prove this we use the inductive argument on $R$. 
From (6) in Proposition \ref{prop:wavepack} we can see that
$\oint_{B(\omega,R^{-1/2} \cap S)} |f_{I,\tau, j}^{\sharp}|^2 \lesssim \oint_{B(\omega,R^{-1/2}) \cap S} |f_\tau|^2 \lesssim 1$.
Using the induction hypothesis we have  
\[
\int_{B_j} |\mathbf{B}_{C\alpha} [Ef_{I,j}^{\sharp}]|^{p_0} \le C_\epsilon R^{(1-\delta)(\epsilon+\delta_2)} \bigg( \sum_\tau \int_S  |f_{\tau,j}^{\sharp}|^2 \bigg)^{3/2+\epsilon}.
\]
By summing these over $j$,
\begin{equation} \label{sumBj}
\sum_j \int_{B_j} |\mathbf{B}_{C\alpha} [Ef_{I,j}^{\sharp}]|^{p_0} \le C_\epsilon R^{(1-\delta)(\epsilon+\delta_2)} \bigg( \sum_\tau \sum_j \int_S |f_{\tau,j}^{\sharp}|^2 \bigg)^{3/2+\epsilon}.
\end{equation}

Now we estimate $\sum_j \int_S |f_{\tau,j}^{\sharp}|^2 $. For this we use the following geometric lemma about the transverse tubes.

\begin{lem}[Lemma 3.5 in \cite{guth2015restriction}] \label{lem:geo1}
There is a nonnegative constant $C$ such that each tube $T \in \mathbb T$ belongs to at most $M^C=R^{C\delta_1}$ different sets $\mathbb T_{j}^{\sharp}$.
\end{lem}

\noindent Using this we can obtain the following lemma.

\begin{lem}\label{CounTube}
For each $\tau$, 
\begin{equation} \label{degCount2}
\sum_{j} \int |f_{\tau,j}^\sharp|^2 \lesssim M^C \int |f_{\tau}|^2
+ R^{-900}\|f_\tau\|_{L^2(S)}^2.
\end{equation}
\end{lem}
\begin{proof}

From (1) in Proposition \ref{prop:wavepack}, 
\begin{align*}
\int |f_{\tau,j}^\sharp|^2 &\lesssim \int \Big| \sum_{\Omega: \Omega \cap \tau \neq \emptyset} \sum_{T \in \mathbb T_j^\sharp(\Omega) } f_{T} \Big|^2 \\
&\lesssim  \sum_{\Omega: \Omega \cap \tau \neq \emptyset}  \int \Big| \sum_{T \in \mathbb T_j^\sharp(\Omega) }  f_{T} \Big|^2. 
\end{align*}
From (4) in Proposition \ref{prop:wavepack}, we see
\[
\int |f_{\tau,j}^\sharp|^2 \lesssim  \sum_{\Omega: \Omega \cap \tau \neq \emptyset}  \sum_{T \in \mathbb T_j^{\sharp}(\Omega)} \int |  f_{T}|^2 + R^{-950}  \|f_\tau\|_{L^2(S)}^2
\]
and by summing over $j$,
\begin{align*}
\sum_j \int |f_{\tau,j}^{\sharp}|^2 &\lesssim  \sum_{\Omega: \Omega \cap \tau \neq \emptyset} \sum_j  \sum_{T \in \mathbb T_j^{\sharp}(\Omega)} \int |  f_{T}|^2 + R^{-900}  \|f_\tau\|_{L^2(S)}^2\\
&\lesssim \sum_{\Omega: \Omega \cap \tau \neq \emptyset} \sum_{T \in \mathbb T(\Omega)} \sum_{j: T \in \mathbb T_j^{\sharp}(\Omega) } \int |  f_{T}|^2 + R^{-900}  \|f_\tau\|_{L^2(S)}^2.
\end{align*} 
By Lemma \ref{lem:geo1}, 
\begin{align*}
\sum_{j} \int |f_{\tau,j}^\sharp|^2
&\lesssim M^C \sum_{\Omega: \Omega \cap \tau \neq \emptyset} \sum_{T \in \mathbb T(\Omega)} \int |  f_{T}|^2 + R^{-900} \|f_\tau\|_{L^2(S)}^2 \\
&\lesssim M^C \sum_{T \in \mathbb T(\Omega):\omega(T) \in \tau} \int |  f_{T}|^2 + R^{-900} \|f_\tau\|_{L^2(S)}^2.
\end{align*}
By (4) of Proposition \ref{prop:wavepack}, we obtain \eqref{degCount2}.
\end{proof}
We plug \eqref{degCount2} into \eqref{sumBj}. Then we have
\begin{equation*} 
\sum_j \int_{B_j} |\mathbf{B}_{C\alpha} [Ef_{I,j}^{\sharp}]|^{p_0} \le C_\epsilon (R^{(1-\delta)(\epsilon+\delta_2)} M^{C} + R^{-1000})\bigg( \sum_\tau  \int_S |f_{\tau}|^2 \bigg)^{3/2+\epsilon},
\end{equation*}
and by summing over $I$,
\begin{equation*}
\sum_{j,I}\int_{B_j} |\mathbf{B}_{C\alpha} [Ef_{I,j}^{\sharp}]|^{p_0}
\le C_\epsilon (R^{\epsilon + \delta_2}   M^{C} R^{ -\delta(\epsilon+\delta_2)} + R^{-1000}) \Big( \sum_\tau \int_S |f_{\tau}|^2 \Big)^{3/2+\epsilon},
\end{equation*}
because the number of subcollections $I$ is at most $2^{K^2}$ which can be absorbed in $C_\epsilon$.
From \eqref{deltas} and \eqref{PolDeg}, we have \( M^{C} R^{ -\delta(\epsilon+\delta_2)}= R^{C\epsilon^4  -  \epsilon^3 - \epsilon^8} \le R^{-\epsilon^3/2} .\) For a sufficiently large $R$ we obtain \eqref{EstTrans}.

\subsubsection{Estimate for the tangential part}
Until now we reduced the problem by using the inductive argument. In this subsection we will directly estimate the remaining part. The key ingredient is the following geometric estimate. 
 
\begin{lem}[Lemma 3.6 in \cite{guth2015restriction}] \label{lem:numTanTube}
For each $j$, the number of different $\Omega$ with $\mathbb T_{j}^{\flat} \cap \mathbb T(\Omega) \neq \emptyset$ is at most $ R^{1/2+O(\delta)}$.
\end{lem}
This lemma implies that $\mathbb T_j^\flat$ is made up of tubes in only $R^{1/2+O(\delta)}$ different directions.	
To prove \eqref{WEST} we will show that
\[ 
K^{400} \sum_j\int_{B_j \cap W} \mathrm{Bil}(Ef_{j}^{\flat})^{p_0}
\le C_\epsilon R^\epsilon R^{\delta_2}  \Big( \sum_{\tau} \int_S |f_\tau|^2 \Big)^{3/2}.
\]
Since the number of cubes $B_j$ is $\lesssim R^{C\delta}$, there is a cube $B_j$ such that
\[ 
K^{400} \sum_j\int_{B_j \cap W} \mathrm{Bil}(Ef_{j}^{\flat})^{p_0}
\le C_\epsilon R^{C\delta} \int_{B_j \cap W} \mathrm{Bil}(Ef_{j}^{\flat})^{p_0},
\] where $K^{400}$ can be absorbed in $C_\epsilon$.
Since $R^{C\delta} \le R^{\epsilon}$ by \eqref{deltas},
it suffices to show that for each $j$,
\begin{equation} \label{BILEST}
\int_{B_j \cap W} \mathrm{Bil}(Ef_{j}^{\flat})^{p_0} \le C_\epsilon R^{C\delta} \Big( \sum_{\tau} \int_S |f_\tau|^2 \Big)^{3/2}.
\end{equation}

Decompose \( B_j \cap W \) into finer cubes $Q$ of side-length \( R^{1/2} \). 
We define $\mathbb T_{1,Q}^{\flat}$ and $\mathbb T_{2,Q}^{\flat}$ by
\begin{align*}
\mathbb T_{1,Q}^{\flat} &=\{ T \in \mathbb T_{j}^{\flat}: T \cap Q \neq \emptyset,~ \omega(T) \in \tau_1 \}, \\
\mathbb T_{2,Q}^{\flat} &=\{ T \in \mathbb T_{j}^{\flat}: T \cap Q \neq \emptyset,~ \omega(T) \in \tau_2 \}. 
\end{align*}
We first show the orthogonality among the bilinear wave packets  $Ef_{T_1}Ef_{T_2}$ for $T_1 \in \mathbb T_{1,Q}^\flat$ and $T_2 \in \mathbb T_{2,Q}^\flat$.  
We can see that the tubes in $\mathbb T_{1,Q}^\flat \cup \mathbb T_{2,Q}^\flat$ are contained in a $O(R^{1/2+\delta})$-neighborhood of some tangent plane. So, the orthogonal property observed in the proof of the 2-dimensional restriction theorem can be obtained.  
\begin{lem} \label{lem:Cord}
Let us set $F_T  = Ef_T$. 
Suppose that $\tau_1$ and $\tau_2$ satisfy the condition that for any $(\xi_1,\zeta_1) \in \bar\tau_1$ and any $(\xi_2,\zeta_2) \in \bar\tau_2$, 
\begin{equation} \label{cond:vw}
|\xi_1 - \xi_2| \gtrsim K^{-1} \quad \text{and} \quad |\zeta_1 - \zeta_2 | \gtrsim K^{-1}.
\end{equation}
Then for any $Q$ intersecting $B_j \cap W$, 
\begin{equation} \label{squreF}
\int \Big|\sum_{T_1 \in \mathbb T_{1,Q}^{\flat}} \sum_{T_2 \in \mathbb T_{2,Q}^{\flat}} 
F_{T_1} F_{T_2} \Big|^2
\lesssim R^{C\delta} \sum_{T_1 \in \mathbb T_{1,Q}^{\flat}} \sum_{T_2 \in \mathbb T_{2,Q}^{\flat}} 
 \int |F_{T_1} F_{T_2}|^2.
\end{equation}
\end{lem}
\begin{proof}
One can write as
\[ 
\int \Big|\sum_{T_1 \in \mathbb T_{1,Q}^{\flat}} \sum_{T_2 \in \mathbb T_{2,Q}^{\flat}} 
F_{T_1} F_{T_2} \Big|^2 = 
\sum_{T_1, T_1' \in \mathbb T_{1,Q}^{\flat}} \sum_{T_2,T_2'\in \mathbb T_{2,Q}^{\flat}} 
\langle F_{T_1} F_{T_2}, F_{T_1'} F_{T_2'}  \rangle.
\]
By Parseval's identity, 
\[ 
\langle F_{T_1} F_{T_2}, F_{T_1'} F_{T_2'}  \rangle
=\langle \widehat F_{T_1} \ast \widehat F_{T_2}, \widehat F_{T_1'} \ast \widehat F_{T_2'}  \rangle.
\]
We now consider the supports of $\widehat F_{T_1} \ast \widehat F_{T_2}$.
Recall that $Ef$ can be written as $( f_T d\sigma_S)^{\spcheck}$. By (1) of Proposition \ref{prop:wavepack} we have $Ef_T = ( \chi_{3\Omega} f_T d\sigma_S)^{\spcheck}$ provided $T \in \mathbb T(\Omega)$. So, $\widehat{F_T}$ is supported in the $O(R^{-1/2})$-neighborhood of $\omega(T)$. 
Let $\omega(T_j) = (\xi_j,\zeta_j,\xi_j\zeta_j)$ for $j=1,2$.
If the above equation does not vanish, then the following relations are satisfied: 
\begin{align}
\xi_1 + \xi_2 &= \xi_1' + \xi_2' + O(R^{-1/2}), \nonumber \\
\zeta_1 + \zeta_2 &= \zeta_1' + \zeta_2' + O(R^{-1/2}), \nonumber  \\
\xi_1 \zeta_1 +  \xi_2 \zeta_2 &= \xi_1' \zeta_1' +  \xi_2' \zeta_2' + O(R^{-1/2}).
\label{trel} 
\end{align}
From these relations it follows that if $(\xi_1,\zeta_1), (\xi_2,\zeta_2)$ and $(\xi_1',\zeta_1')$ are given, then $(\xi_2',\zeta_2')$ is determined as follows:
\begin{equation} \label{vvv}
(\xi_2',\zeta_2')=(\xi_1+\xi_2-\xi_1',\zeta_1+\zeta_2-\zeta_1')+O(R^{-1/2}).
\end{equation}
So, if we set $\omega_2' = \omega(T_1) + \omega(T_2) - \omega(T_1')$, then 
\[ 
\sum_{T_1, T_1' \in \mathbb T_{1,Q}^{\flat}} \sum_{T_2,T_2'\in \mathbb T_{2,Q}^{\flat}} 
\langle F_{T_1} F_{T_2}, F_{T_1'} F_{T_2'}  \rangle
= \sum_{T_1, T_1' \in \mathbb T_{1,Q}^{\flat}} \sum_{T_2  \in \mathbb T_{2,Q}^{\flat}} \sum_{\substack{T_2'  \in \mathbb T_{2,Q}^{\flat}: \\ |\omega(T_2')-  \omega_2'| 
\lesssim R^{-1/2} }} \langle F_{T_1} F_{T_2}, F_{T_1'} F_{T_2'}  \rangle.
\]
Note that the number of choice of $T_2'$ is $O(1)$, because $T_2'$ passes through $Q$.

Now we restrict $(\xi_1,\zeta_1)$ when $(\xi_1',\zeta_1')$ and $(\xi_2,\zeta_2)$ are given. For this we insert \eqref{vvv} into \eqref{trel}. By rearranging this, we obtain
\begin{equation} \label{con:line}
(\xi_1 - \xi_1')(\zeta_2 - \zeta_1') + (\zeta_1-\zeta_1')(\xi_2 - \xi_1') = O(R^{-1/2}).
\end{equation}
Let $\ell(T_1',T_2)$ be the line passing through $(\xi_1',\zeta_1')$ and  of direction normal to the vector $(\zeta_2 - \zeta_1', \xi_2 - \xi_1')$. Then
from \eqref{con:line} it follows that if $(\xi_1',\zeta_1')$ and $(\xi_2,\zeta_2)$ are given, then  $(\xi_1,\zeta_1)$ is contained in a $O(R^{-1/2})$-neighborhood of the line $\ell(T_1',T_2)$. 
Thus, it implies
\begin{multline*}
\sum_{T_1, T_1' \in \mathbb T_{1,Q}^{\flat}} \sum_{T_2,T_2'\in \mathbb T_{2,Q}^{\flat}} 
\langle F_{T_1} F_{T_2}, F_{T_1'} F_{T_2'}  \rangle \\
= \sum_{T_1' \in \mathbb T_{1,Q}^{\flat}} \sum_{T_2 \in \mathbb T_{2,Q}^{\flat}}  \sum_{\substack{T_1 \in \mathbb T_{1,Q}^{\flat} : \\ \dist(\omega(T_1),\ell(T_1',T_2)) \lesssim R^{-1/2}}} \sum_{\substack{T_2'  \in \mathbb T_{2,Q}^{\flat}: \\ |\omega(T_2')-  \omega_2'| 
		\lesssim R^{-1/2} }}
\langle F_{T_1} F_{T_2}, F_{T_1'} F_{T_2'}  \rangle.
\end{multline*}

We see that all tube segments $T \cap B(R)$ for $T \in \mathbb T_{1,Q}^\flat \cup \mathbb T_{2,Q}^\flat$ are contained in the $R^{1/2+C\delta}$-neighborhood of some plane.
So, all directions 
$v(T)$ for $T \in \mathbb T_{1,Q}^\flat \cup \mathbb T_{2,Q}^\flat$ are also contained in the $R^{-1/2+C\delta}$-neighborhood of a plane $\pi$ passing through the origin. 
If $\omega(T)=(\xi,\zeta,\xi \zeta)$ then we can see 
\[
v(T) = \frac{1}{\sqrt{\xi^2+\zeta^2+1}}(\zeta,\xi,-1).
\] 
So, we can correspond $v(T)$, $T \in \mathbb T_{1,Q}^\flat \cup \mathbb T_{2,Q}^\flat$, to a point $(\zeta,\xi,-1)$.
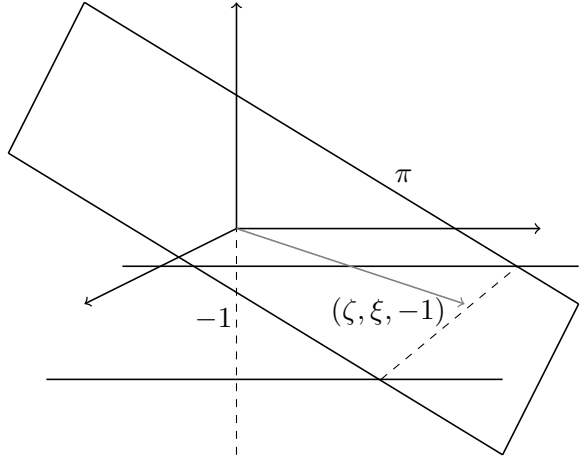
\begin{center}
\begin{tikzpicture}[scale = 1]
\draw[->,semithick] (0,0) -- (4,0);
\draw[->,semithick] (0,0) -- (-2,-1);
\draw[->,semithick] (0,0) -- (0,3);
\draw[dashed] (0,-3) -- (0,0);
\draw[semithick] (-1.5,-0.5) -- (4.5,-0.5);
\draw[semithick] (-2.5,-2) -- (3.5,-2);
\draw (-0.3,-1.2) node {$-1$};
\draw[semithick] (-3,1) -- (3.5,-3);
\draw[semithick] (-3+1,1+2) -- (3.5+1,-3+2);
\draw[semithick] (-3,1) -- (-3+1,1+2);
\draw[semithick] (3.5,-3) -- (3.5+1,-3+2);
\draw (2.2,0.7) node {$\pi$};
\draw[dashed] (1.9,-2) -- (3.7,-0.5);
\draw (2,-1.1) node {$(\zeta,\xi,-1)$};
\draw[->,semithick,color=gray] (0,0) -- (3,-1);
\end{tikzpicture}
\end{center}
By considering the mapping 
\[
\frac{1}{\sqrt{\xi^2+\zeta^2+1}}(\zeta,\xi,-1) \to (\xi,\zeta),
\]
it follows that $(\xi_1,\zeta_1)$ is in the $O(R^{-1/2+\delta})$-neighborhood of the line $l$  passing through $(\xi_1',\zeta_1')$ and $(\xi_2,\zeta_2)$.
On the other hand, $(\xi_1,\zeta_1)$ obeys \eqref{con:line}. Thus, $(\xi_1,\zeta_1)$ is contained in the intersection between the  $R^{-1/2+C\delta}$-neighborhood of $l$ and the $O(R^{-1/2})$-neighborhood of $\ell(T_1',T_2)$.

Now we consider the directions of $l$ and $\ell(T_1',T_2)$. The direction of $l$ is normal to the vector $(\zeta_2-\zeta_1', - \xi_2+ \xi_1')$, and that of $\ell(T_1',T_2)$ is normal to $(\zeta_2-\zeta_1', \xi_2-\xi_1')$.
The condition \eqref{cond:vw} guarantees that the two vectors $(\zeta_2-\zeta_1', - \xi_2+ \xi_1')$ and $(\zeta_2-\zeta_1', \xi_2-\xi_1')$ are transverse. This means that the directions of $l$ and $\ell(T_1',T_2)$ are also transverse, and thus $(\xi_1,\zeta_1)$ is contained in a disc of radius $R^{-1/2+C\delta}$. Since all tubes are passing through $Q$, we conclude that the number of choice of $T_1$ is $O(R^{C\delta})$. 

Accordingly, it follows that
\[
\sum_{T_1, T_1' \in \mathbb T_{1,Q}^{\flat}} \sum_{T_2,T_2'\in \mathbb T_{2,Q}^{\flat}} 
\langle F_{T_1} F_{T_2}, F_{T_1'} F_{T_2'}  \rangle 
\lesssim  R^{C\delta} \sum_{T_1 \in \mathbb T_{1,Q}^{\flat}} \sum_{T_2 \in \mathbb T_{2,Q}^{\flat}} 
\int |F_{T_1} F_{T_2}|^2,
\]
which is \eqref{squreF}.
\end{proof}

Due to the orthogonality we can obtain the following estimate.

\begin{lem} \label{sppE}
Let $\tau_1$ and $\tau_2$ be as in Lemma \ref{lem:Cord}. Then for any $Q$ intersecting $B_j \cap W$, 
\begin{multline} \label{Qest}
\int_Q |Ef_{\tau_1,j}^{\flat}|^2 |Ef_{\tau_2,j}^{\flat}|^2 \lesssim R^{C\delta} R^{-1/2} \Big( \sum_{T_1 \in \mathbb T_{1,Q}^{\flat} } \|f_{T_1}\|_{L^2(S)}^2  \Big) \Big( \sum_{T_2 \in \mathbb T_{2,Q}^{\flat} } \|f_{T_2}\|_{L^2(S)}^2  \Big) \\+ O(R^{-990}\|f\|_{L^2(S)}^2).
\end{multline}
\end{lem}
\begin{proof}
By (3) of Proposition \ref{prop:wavepack}, we have that for $k=1,2$,
\[
Ef_{\tau_k,j}^{\flat} = \sum_{T \in \mathbb T_{k,Q}^{\flat}} Ef_T + O(R^{-990}\|f\|_{L^2(S)}).
\] 
Substituting this in the left-hand side of \eqref{Qest} and applying Lemma \ref{lem:Cord}, we get
\begin{equation} \label{cord}
\int_Q |Ef_{\tau_1,j}^{\flat}|^2 |Ef_{\tau_2,j}^{\flat}|^2 \le 
R^{C\delta} \sum_{T_1\in \mathbb T_{1,Q}^{\flat}} \sum_{ T_2\in \mathbb T_{2,Q}^{\flat}}  \int |Ef_{T_1} Ef_{T_2} |^2 + O(R^{-1500}\|f\|_{L^2(S)}^4),
\end{equation}
where we use a trivial estimate $\|Ef\|_{L^2(B(R))} \lesssim R^C\|f\|_{L^2(S)}$ (or \eqref{Est22} below).
Now it suffices to show that
\begin{equation} \label{elementEST}
\int |Ef_{T_1} Ef_{T_2} |^2 \lesssim R^{-1/2} \|f_{T_1}\|_{L^2(S)}^2 \|f_{T_2}\|_{L^2(S)}^2.
\end{equation}
By the Plancherel theorem and (1) of Proposition \ref{prop:wavepack},
\begin{equation*}
\int |Ef_{T_1} Ef_{T_2} |^2 = 
\int |f_{T_1}d\sigma_{3\Omega_1} \ast f_{ T_2}d\sigma_{3\Omega_2}|^2.
\end{equation*}
From the condition \eqref{cond:vw}, we see that the unit normal vectors $n(\Omega_1)$ and $n(\Omega_2)$ are transverse, and thus the translations of $\Omega_1$ meet $\Omega_2$ transversally. From this observation it follows that \( \| d\sigma_{3\Omega_1} \ast d\sigma_{3\Omega_2} \|_\infty \lesssim R^{-1/2} \). 
Using this we have
\[
\| f_{T_1}d\sigma_{3\Omega_1}  \ast f_{ T_2}d\sigma_{3\Omega_2} \|_\infty \le CR^{-1/2}\|f_{T_1}\|_{L^\infty(S)} \|f_{T_2}\|_{L^\infty(S)}.
\]
On the other hand, by Young's inequality it gives
\[
\| f_{T_1}d\sigma_{3\Omega_1}  \ast f_{T_2}d\sigma_{3\Omega_2} \|_1 \le \|f_{T_1}\|_{L^1(S)} \|f_{T_2}\|_{L^1(S)}.
\]
By interpolating these two estimates, 
\[
\|f_{T_1}d\sigma_{3\Omega_1}   \ast f_{ T_2}d\sigma_{3\Omega_2} \|_2 \lesssim R^{-1/4} \| f_{ T_1}\|_{L^2(S)} \| f_{T_2}\|_{L^2(S)},
\]
and so
\[ 
\| Ef_{T_1} Ef_{T_2} \|_2 \lesssim R^{-1/4} \| f_{ T_1}\|_{L^2(S)} \| f_{T_2}\|_{L^2(S)},
\]
which is \eqref{elementEST}.
\end{proof}

Now we sum \eqref{Qest} over $Q$. For this we use the way of dealing with two dimensional Kakeya set (see \cite{cordoba1982geometric}).
By simple calculation we know that if $T_1 \in \mathbb T_{1,Q}^\flat$ and $T_2 \in \mathbb T_{2,Q}^{\flat}$ then 
\begin{equation} \label{tritranstube}
\int  \chi_{T_1} \chi_{T_2} \sim K R^{3/2 + 3\delta},
\end{equation}
because we have $|v(T_1) - v(T_2)| \sim K^{-1}$ by \eqref{cond:vw}.

Inserting $K^{-1}R^{-3/2} \int \chi_{T_1} \chi_{T_2}$ into the right-hand side of \eqref{Qest}, we have
\begin{multline*}
\int_Q |Ef_{\tau_1,j}^{\flat}|^2 |Ef_{\tau_2,j}^{\flat}|^2  \\
\lesssim K^{-1} R^{C\delta} R^{-2}  \sum_{T_1 \in \mathbb T_{1,Q}^{\flat} } \sum_{T_2 \in \mathbb T_{2,Q}^{\flat} } \|f_{T_1}\|_{L^2(S)}^2    \|f_{T_2}\|_{L^2(S)}^2  \int  \chi_{T_1} \chi_{T_2}+ O(R^{-1500}\|f\|_{L^2(S)}^4).
\end{multline*}
Using $\int \chi_{T_1} \chi_{T_2} \lesssim R^{C\delta} \int_{KQ} \chi_{T_1} \chi_{T_2}$ we rewrite this as 
\begin{multline*}
\int_Q |Ef_{\tau_1,j}^{\flat}|^2 |Ef_{\tau_2,j}^{\flat}|^2 \\
\lesssim K^{-1} R^{C\delta} R^{-2}  \sum_{T_1 \in \mathbb T_{1,j}^{\flat} } \sum_{T_2 \in \mathbb T_{2,j}^{\flat} } \|f_{T_1}\|_{L^2(S)}^2    \|f_{T_2}\|_{L^2(S)}^2  \int_{KQ}  \chi_{T_1} \chi_{T_2}+ O(R^{-1500}\|f\|_{L^2(S)}^4)
\end{multline*}
where
$
\mathbb T_{1,j}^{\flat} =\{ T \in \mathbb T_{j}^{\flat}:  \omega(T) \in \tau_1 \}$ and
$\mathbb T_{2,j}^{\flat} =\{ T \in \mathbb T_{j}^{\flat}: \omega(T) \in \tau_2 \}$. 
\smallskip
\\
We sum the above estimate over $Q$ with $Q \cap B_j \cap W \neq \emptyset$, and insert \eqref{tritranstube}. Then
\begin{align*}
\int_{B_j \cap W} & |Ef_{\tau_1,j}^{\flat}|^2 |Ef_{\tau_2,j}^{\flat}|^2 \\
&\lesssim  K^C  R^{C\delta} R^{-2}\sum_{T_1 \in \mathbb T_{1,j}^{\flat} } \sum_{T_2 \in \mathbb T_{2,j}^{\flat} } \|f_{T_1}\|_{L^2(S)}^2    \|f_{T_2}\|_{L^2(S)}^2  \int  \chi_{T_1} \chi_{T_2}+ O(R^{-1000}\|f\|_{L^2(S)}^4)\\
&\lesssim K^C R^{C\delta} R^{-1/2}  \Big( \sum_{T_1 \in \mathbb T_{1,j}^{\flat} } \|f_{T_1}\|_{L^2(S)}^2 \Big) \Big(  \sum_{T_2 \in \mathbb T_{2,j}^{\flat} }    \|f_{T_2}\|_{L^2(S)}^2  \Big) + O(R^{-1000}\|f\|_{L^2(S)}^4) \\
&\lesssim K^C R^{C\delta} R^{-1/2}   \|f_{\tau_1,j}^\flat\|_{L^2(S)}^2  \|f_{\tau_2,j}^\flat\|_{L^2(S)}^2   + O(R^{-900}\|f\|_{L^2(S)}^4). 
\end{align*}
Here, the (5) of Proposition \ref{prop:wavepack} is used in the last line. 
So, it implies
\begin{align*}
\int_{B_j \cap W} \mathrm{Bil}(Ef_j^\flat)^4 
&\lesssim K^C R^{C\delta} R^{-1/2}   
\sum_{\tau_1}  \|f_{\tau_1,j}^\flat\|_{L^2(S)}^2  \sum_{\tau_2} \|f_{\tau_2,j}^\flat\|_{L^2(S)}^2 + K^2 R^{-900}\|f\|_{L^2(S)}^4 \\
&\lesssim K^C R^{C\delta} R^{-1/2}   
\Big( \sum_{\tau}  \|f_{\tau,j}^\flat\|_{L^2(S)}^2 \Big)^2 + K^2R^{-900}\|f\|_{L^2(S)}^4. 
\end{align*}
Therefore, we obtain
\begin{equation} \label{bilEST1}
\| \mathrm{Bil}(Ef_{j}^{\flat})\|_{L^4(B_j \cap W)} 
\le C_\epsilon \big( R^{C\delta}  R^{-1/8} \Big( \sum_{\tau}\|f_{\tau,j}^{\flat}\|_{L^2(S)}^2 \Big)^{1/2}+ R^{-200} \|f\|_{L^2(S)} \big).
\end{equation}

On the other hand, by Plancherel's theorem it follows that
\[
\int_{-R/2}^{R/2}\int_{D(R)} |Ef(x',x_3)|^2 dx' dx_3  \lesssim R \| f\|_{L^2(S)}^2.
\]
This is written as
\begin{equation} \label{Est22}
\| Ef \|_{L^2(B(R))} \lesssim R^{1/2}  \| f\|_{L^2(S)}.
\end{equation}
By the Cauchy-Schwarz inequality and the above estimate,
\[
\| \mathrm{Bil}(Ef_{j}^{\flat}) \|_{L^2(B_j \cap W)} ^2 
\lesssim K^C \sum_{\tau_1} \sum_{\tau_2} \int_{B(R)} |Ef_{\tau_1,j}^{\flat}| |Ef_{\tau_2,j}^{\flat}| 
\lesssim K^{C'} R  \sum_{\tau} \|f_{\tau,j}^{\flat}\|_{L^2(S)}^2 .
\]
Thus,
\begin{equation} \label{bilEST2}
\| \mathrm{Bil}(Ef_{j}^{\flat}) \|_{L^2(B_j \cap W)} \le C_\epsilon R^{1/2} (\sum_{\tau} \|f_{\tau,j}^{\flat}\|_{L^2(S)}^2 )^{1/2}.
\end{equation}
We now interpolate \eqref{bilEST1} with \eqref{bilEST2}. Then it follows that for all \( 3 \le p \le 4, \)
\begin{equation} \label{resInt}
\| \mathrm{Bil}(Ef_{j}^{\flat})\|_{L^p(B_j \cap W)} 
\le C_\epsilon R^{C\delta}  R^{\frac{5}{2p} - \frac{3}{4}} \Big( \sum_{\tau}\|f_{\tau,j}^{\flat}\|_{L^2(S)}^2 \Big)^{1/2}+ O(R^{-100} \|f\|_{L^2(S)}).
\end{equation}

Let us consider $\sum_{\tau}\|f_{\tau,j}^{\flat}\|_{L^2(S)}^2$. We utilize a geometric estimate for the tangential tubes.
By (4) of Proposition \ref{prop:wavepack},  
\begin{equation*}
\sum_{\tau}  \int_S |f_{\tau,j}^{\flat}|^2 \lesssim \sum_{T \in \mathbb T_{j}^\flat} \int_S |f_{T}|^2 +  R^{-900} \sum_{\tau} \|f_{\tau} \|_{L^2(S)}^2.
\end{equation*}
By Lemma \ref{lem:numTanTube}, there is an $\Omega$ such that
\begin{align*}
\sum_{\tau}  \int_S |f_{\tau,j}^{\flat}|^2 
&\lesssim R^{1/2+C\delta} \sum_{T \in \mathbb T_{j}^\flat(\Omega)} \int_S |f_{T}|^2 + R^{-900} \sum_{\tau} \|f_{\tau} \|_{L^2(S)}^2 \\
&\lesssim R^{1/2+C\delta} \sum_{T \in \mathbb T(\Omega)} \int_S |f_{T}|^2 + R^{-900} \sum_{\tau} \|f_{\tau} \|_{L^2(S)}^2.
\end{align*}
Thus, by (5) of Proposition \ref{prop:wavepack} we have
\begin{align*}
\sum_{\tau}  \int_S |f_{\tau,j}^{\flat}|^2
\lesssim R^{1/2+C\delta} \int_{\Omega} |f|^2 + R^{-900} \sum_{\tau} \|f_{\tau} \|_{L^2(S)}^2. 
\end{align*}
From $
\oint_{\Omega} |f|^2 \lesssim 1$, we obtain
\begin{equation} \label{tangGeoEst}
\sum_{\tau}  \int_S |f_{\tau,j}^{\flat}|^2  \lesssim R^{1/2+C\delta} R^{-1} = R^{-1/2+C\delta}.
\end{equation}

By raising both sides of \eqref{resInt} to the $p$th power,
\[
\| \mathrm{Bil}(Ef_{j}^{\flat})\|_{L^p(B_j \cap W)}^p 
\lesssim C_\epsilon R^{C\delta}  R^{\frac{5}{2} - \frac{3p}{4}} \Big( \sum_{\tau}\|f_{\tau,j}^{\flat}\|_{L^2(S)}^2 \Big)^{p/2}+ R^{-300} \|f\|^p_{L^2(S)}.
\]
Using \eqref{tangGeoEst} we have the estimate
\begin{align*}
\Big( \sum_{\tau} \|f_{\tau,j}^{\flat}\|_{L^2(S)}^2 
\Big)^{p/2} 
&= \Big( \sum_{\tau} \|f_{\tau,j}^{\flat}\|_{L^2(S)}^2 
\Big)^{(p-3)/2} \Big( \sum_{\tau} \|f_{\tau,j}^{\flat}\|_{L^2(S)}^2 
\Big)^{3/2} \\
&\lesssim  R^{C\delta} R^{\frac{3}{4} - \frac{p}{4}} \Big( \sum_{\tau} \|f_{\tau,j}^{\flat}\|_{L^2(S)}^2 
\Big)^{3/2}.
\end{align*}
Therefore, it gives
\begin{align*}
\| \mathrm{Bil}(Ef_{j}^{\flat})\|_{L^{p}(B_j \cap W)}^{p} 
&\lesssim C_\epsilon R^{C\delta}  R^{\frac{13}{4} - {p}} \Big( \sum_{\tau}\|f_{\tau}\|_{L^2(S)}^2 \Big)^{3/2}+ R^{-300} \|f\|_{L^2(S)}^p. \\
&\lesssim C_\epsilon R^{C\delta}  R^{\frac{13}{4} - {p}} \|f\|_{L^2(S)}^3,
\end{align*}
where we used the estimate $\|f\|_{L^2(S)} \lesssim 1$.
By taking $p=p_0$ we finally obtain \eqref{BILEST}.

\section{Appendix}

In this section, we prove the wave packet decomposition.

\subsection{Proof of Proposition \ref{prop:wavepack}}

We first define bump functions. Let $\phi:\mathbb R^2 \mapsto \mathbb R$ be a nonnegative Schwartz function such that $\widehat \phi$ is supported in a disc $D(0,\frac{3}{2})$ and $\sum_{j \in \mathbb Z^2} \phi(x-j) =1$. Let $\psi:\mathbb R^2 \mapsto \mathbb R$ is a nonnegative smooth function that is equal to 1 on a disc $D(0,2)$ and supported in $D(0,3)$.
For a disc $D$ we define $a_D$ to be an affine map from the unit disc to $D$, and let $\phi_D = \phi \circ a_D^{-1}$ and $\psi_D = \psi \circ a_D^{-1}$.

By translating we may assume that $B(R)$ is centered at the origin.
We define $f_\Omega := f \chi_\Omega$ and  $\tilde f_{\bar\Omega} := \tilde f \chi_{\bar\Omega}$ where $\bar \Omega = \{\xi' \in \mathbb R^2 : (\xi',\xi_3) \in \Omega \}$. Then we have $\tilde f_{\bar\Omega}(\xi') = J(\xi') f_\Omega(\xi)$ where $J(\xi_1,\xi_2) := \frac{d\sigma(\xi_1,\xi_2)}{d\xi_1d\xi_2}$. Since $|J| \sim 1$ on $B(R)$, we may identify $\tilde f_{\bar\Omega}$ with $f_\Omega$.
For each $T \in \mathbb T(\Omega)$ we define $\tilde f_T$ by
\begin{equation}\label{def:f_T}
\tilde{f}_T := \psi_{\bar \Omega}(\widehat\phi_{D_T}*\tilde{f}_{\bar \Omega})
\end{equation}
where  $D_T = \{ x' \in \mathbb R^2 : (x',0) \in T\}$. From this definition it follows that $\tilde f_T$ is supported in $3\bar\Omega$.
We define $f_T$ on $S$ by $\tilde f_T(\xi') = J(\xi') f_T(\xi',\xi_1\xi_2)$. Then  $f_T$ has Property (1).

Consider Property (2). For $T \in \mathbb T(\Omega)$, we write 
\begin{align}
Ef_{T}(x',x_3) &= \int_{Q(1)} e^{ 2\pi i(x'\cdot \xi'+x_3 \xi_1\xi_2)} \tilde f_T(\xi') d \xi' \nonumber \\
&=\int e^{2\pi ix' \cdot \xi'} \Psi_{x_3}(\xi') (\widehat\phi_{D_T}*\tilde{f}_{\bar \Omega})(\xi') d \xi' \nonumber \\
&= \Psi_{x_3}^{\spcheck} \ast (\phi_{D_T}  \tilde f_{\bar \Omega}^{\spcheck})(x')  
\label{EfForm}
\end{align}
where $\Psi_{x_3}(\xi') := e^{2\pi i(\xi_1\xi_2 x_3)} \psi_{\bar \Omega}(\xi')$.

If $(\omega_1, \omega_2)$ be the center of $\bar \Omega$ then we can see that the normal vector $n(\Omega)$ is parallel to $(\omega_2,\omega_1,-1)$, so the tubes $T \in \mathbb T(\Omega)$ are written as
\begin{equation} \label{tubeEq}
T = \{ (x',x_3) :  | x'- x'_T + x_3 (\omega_2,\omega_1) | \lesssim R^{1/2+ \delta}  \}
\end{equation}
where $x'_T$ is the center of $D_T$.

Using integrating by parts, we can obtain that for $(x',x_3) \in \mathbb R^2 \times [-10R,10R]$,
\[
|\Psi^{\spcheck}_{x_3}(x')| \le C_M |\bar \Omega| (1+R^{-1/2}|x'+ x_3 (\omega_2,\omega_1)  |)^{-M}, \qquad \forall M > 0.
\]
By inserting this into \eqref{EfForm} we have that for $(x',x_3) \in \mathbb R^2 \times [-10R,10R]$,
\begin{align}
|E f_T(x',x_3)| &\le C_M |\bar \Omega| (1+R^{-1/2}|x'-x'_T+ x_3 (\omega_2,\omega_1) |)^{-M}
\int | \phi_{D_T} \tilde f_{\bar\Omega}^{\spcheck}(y') | dy' \nonumber\\
& \le C_M |\bar \Omega| |D_T|^{1/2} (1+R^{-1/2}|x'-x'_T+ x_3 (\omega_2,\omega_1) |)^{-M} \| \tilde f_{\bar\Omega} \|_2 \nonumber\\
&\le C_M R^{-1/2 +\delta} (1+R^{-1/2}|x'-x'_T+ x_3 (\omega_2,\omega_1) |)^{-M} \| f \|_{L^2(\Omega)}
\label{EfTest}
\end{align}
for any $M >0$.
Using this and \eqref{tubeEq} we have that for $x \in B(R) \setminus T$,
\[
|Ef(x)| \le C_M R^{-1/2+\delta} R^{-\delta M} \|f\|_{L^2(\Omega)} \lesssim R^{-1000} \|f\|_{L^2(\Omega)}
\]
provided $M>0$ is sufficiently large. Thus we have  Property  (2).

Let $\widetilde{\mathbb T}(\Omega)$ be the collection of cylindrical tubes of radius $R^{1/2+\delta}$ which are parallel to $n(\Omega)$ and cover $\mathbb R^3$ with finite overlap. By the definition \eqref{def:f_T} it follows that
\[
\tilde f = \sum_{\Omega} \sum_{T \in \widetilde{\mathbb T}(\Omega)} \tilde f_T.
\]
Since $Ef$ is a linear operator, we have
\[
Ef(x) =  \sum_{\Omega}\sum_{T \in \widetilde{\mathbb T}(\Omega)} Ef_{T}(x).
\]
If $x \in B(R)$ we can restrict $T \in \widetilde{\mathbb T}(\Omega)$ to $\mathbb T(\Omega)$ in the above summation, because by \eqref{EfTest} the contribution of $Ef_T$ for $T \in \widetilde{\mathbb T} \setminus \mathbb T$ is negligible. Thus, we have Property (3).

Consider Property (4). Let $T_1$ and $T_2$ be two disjoint tubes in $\mathbb T(\Omega)$. We see that $\big| \int f_{T_1} \overline{f_{T_2}} \big| \lesssim \big| \int \tilde f_{T_1}\overline{\tilde f_{T_2}} \big|$. We will use that $\tilde f_T$ is essentially Fourier supported in $2D_T$. By Parseval's identity, 
\begin{align}
\int \tilde f_{T_1}\overline{\tilde f_{T_2}} &= \int \big( {\psi_{\bar\Omega}^2}^{\spcheck} * ({\phi_{D_{T_1}}}{\tilde{f}}^{\spcheck}_{\bar\Omega}) \big) \overline{({\phi_{D_{T_2}}}{ \tilde{f}}_{\bar\Omega}^{\spcheck})} \nonumber\\
&\le \| \tilde{f}^{\spcheck}_{\bar\Omega} \|_{\infty}^2 \bigg| \int \big( {\psi_{\bar\Omega}^2}^{\spcheck} * \phi_{D_{T_1}} ) \overline{ \phi_{D_{T_2}} } \bigg| .
\label{ppoth}
\end{align}
Since $\phi$ and $\psi$ are smooth bump functions, we have $|\phi_{D}(x')| \lesssim  (1+\dist(x', D))^{-4000} $ and $|\psi_\Omega^{\spcheck}(x')| \lesssim (1+ R^{-1/2}|x'|)^{-4000}$.
Since $T_1$ and $T_2$ are disjoint, the distance between $D_{T_1}$ and  $D_{T_2}$ is $ \ge (1/4)R^{1/2+\delta}$, from which we have $\Big| \int \big({\psi_{\bar\Omega}^2}^{\spcheck} * \phi_{D_{T_1}} ) \overline{ \phi_{D_{T_2}} } \Big| \lesssim R^{-2000}$.
Using the Riemann-Lebesgue lemma and H\"older's inequality we have  $\| \tilde{f}^{\spcheck}_{\bar\Omega} \|_{\infty} \lesssim \|f\|_{L^2(\Omega)}$. 
Thus by inserting these estimates into \eqref{ppoth} we have Property (4).

Consider Property (5). It easily follows from  Property (4). Indeed, by (4),
\[
\sum_{T \in \mathbb T(\Omega)} \int_{S} |f_T|^2 \lesssim 
\int_{S} \Big| \sum_{T \in \mathbb T(\Omega)} f_T \Big|^2 + R^{-900}\int_{\Omega}|f|^2.
\]
Since
\(
\int_{S} \Big| \sum_{T \in \mathbb T(\Omega)} f_T \Big|^2 
\lesssim \int_S |f_\Omega|^2 = \int_{\Omega} |f|^2,
\)
we have Property (5).

Consider Property (6).
Let us denote by $V = \{ x' \in \mathbb R^2 : (x',x_3) \in B(\omega,R^{-1/2}) \cap S \}$. By the relation between $f$ and $\tilde f$ it suffices to show
\[
\Big\| \sum_{T \in \mathbb T': \omega(T) \in \tau} \tilde f_T \Big\|_{L^2(V)} \lesssim \| \tilde f_{\bar \tau}  \|_{L^2(\tilde V)}
\] where $\widetilde V = \{ x' \in \mathbb R^2 : (x',x_3) \in 10B(\omega,R^{-1/2}) \cap S \}$.
By Property (1),
\[
\Big\| \sum_{T \in \mathbb T': \omega(T) \in \tau} \tilde f_T \Big\|_{L^2(B)} \le \Big\| \sum_{\Omega: 3\bar\Omega \cap V \cap \bar\tau \neq \emptyset}\sum_{T \in \mathbb T'(\Omega)} \tilde f_T \Big\|_2.
\]
Since $\Omega$ are finitely overlapped, we have
\[
\Big\| \sum_{\Omega: 3\bar\Omega \cap V  \cap \bar \tau \neq \emptyset}\sum_{T \in \mathbb T'(\Omega)} \tilde f_T \Big\|_2^2 \lesssim 
\sum_{\Omega: 3\bar\Omega \cap V \cap \bar\tau \neq \emptyset} \Big\| \sum_{T \in \mathbb T'(\Omega)} \tilde f_T \Big\|_2^2.
\]
By Property (4) this is bounded by
\[
\lesssim \sum_{\Omega: 3\bar\Omega \cap V \cap \bar\tau  \neq \emptyset} \sum_{T \in \mathbb T'(\Omega)} \| \tilde f_T \|_2^2 + R^{-900} \sum_{\Omega: 3\bar\Omega \cap V  \cap \bar\tau \neq \emptyset} \| \tilde f \|_{L^2(\bar \Omega)}^2.
\]
Now we replace $\mathbb T'$ with $\mathbb T$. Then the above is bounded by 
\[
\sum_{\Omega: 3\bar\Omega \cap V \cap \bar\tau  \neq \emptyset} \sum_{T \in \mathbb T(\Omega)} \| \tilde f_T \|_2^2 + R^{-900} \sum_{\Omega: 3\bar\Omega \cap V  \cap \bar\tau \neq \emptyset} \| \tilde f \|_{L^2(\bar\Omega)}^2. 
\]
By using (4) again, this is 
\begin{align*}
&\lesssim \sum_{\Omega: 3\bar\Omega \cap V \cap \bar\tau  \neq \emptyset} \Big\|  \sum_{T \in \mathbb T(\Omega)}  \tilde f_T \Big\|_2^2 + R^{-900} \sum_{\Omega: 3\bar\Omega \cap V  \cap \bar\tau \neq \emptyset} \| \tilde f \|_{L^2(\bar\Omega)}^2 \\
&\lesssim \sum_{\Omega: 3\bar\Omega \cap V \cap \bar\tau \neq \emptyset} \| \tilde f \|_{L^2(\bar\Omega)}^2 \\
&\lesssim \| \tilde f_{\bar\tau} \|_{L^2(\widetilde V)}^2.
\end{align*}
Thus we have Property (6).


\begin{bibdiv}
\begin{biblist}

\bib{bennett2006multilinear}{article}{
      author={Bennett, J.},
      author={Carbery, A.},
      author={Tao, T.},
       title={On the multilinear restriction and {Kakeya} conjectures},
        date={2006},
     journal={Acta mathematica},
      volume={196},
      number={2},
       pages={261\ndash 302},
}

\bib{bourga1991besicovitch}{article}{
      author={Bourgain, J.},
       title={Besicovitch type maximal operators and applications to {Fourier}
  analysis},
        date={1991},
     journal={Geometric and Functional analysis},
      volume={1},
      number={2},
       pages={147\ndash 187},
}

\bib{bourgain2000harmonic}{article}{
      author={Bourgain, J.},
       title={Harmonic analysis and combinatorics: how much may they contribute
  to each other},
organization={Amer. Math. Soc.},
        date={2000},
     journal={Mathematics: Frontiers and Perspectives},
       pages={13\ndash 32},
}

\bib{bourgain2011bounds}{article}{
      author={Bourgain, J.},
      author={Guth, L.},
       title={Bounds on oscillatory integral operators based on multilinear
  estimates},
        date={2011},
     journal={Geometric and Functional Analysis},
      volume={21},
      number={6},
       pages={1239\ndash 1295},
}

\bib{cordoba1982geometric}{inproceedings}{
      author={C{\'o}rdoba, A.},
       title={Geometric fourier analysis},
        date={1982},
   booktitle={Annales de l'institut fourier},
      volume={32},
       pages={215\ndash 226},
}

\bib{guth2010endpoint}{article}{
      author={Guth, L.},
       title={The endpoint case of the {Bennett}--{Carbery}--{Tao} multilinear
  {Kakeya} conjecture},
        date={2010},
     journal={Acta mathematica},
      volume={205},
      number={2},
       pages={263\ndash 286},
}

\bib{guth2015restriction}{article}{
      author={Guth, L.},
       title={A restriction estimate using polynomial partitioning},
        date={2015},
     journal={Journal of the American Mathematical Society},
       pages={http://dx.doi.org/10.1090/jams827},
}

\bib{guth2015short}{article}{
      author={Guth, L.},
       title={A short proof of the multilinear {Kakeya} inequality},
organization={Cambridge Univ Press},
        date={2015},
     journal={Mathematical Proceedings of the Cambridge Philosophical Society},
      volume={158},
      number={1},
       pages={147\ndash 153},
}

\bib{guth2010erdos}{article}{
      author={Guth, L.},
      author={Katz, N.},
       title={On the {Erd\"os} distinct distance problem in the plane},
        date={2015},
     journal={Annals of Mathematics},
      volume={181},
      number={2},
       pages={155\ndash 190},
}

\bib{kaplan2012simple}{article}{
      author={Kaplan, H.},
      author={Matou{\v{s}}ek, J.},
      author={Sharir, M.},
       title={Simple proofs of classical theorems in discrete geometry via the
  {Guth--Katz} polynomial partitioning technique},
        date={2012},
     journal={Discrete \& Computational Geometry},
      volume={48},
      number={3},
       pages={499\ndash 517},
}

\bib{lee2006bilinear}{article}{
      author={Lee, S.},
       title={Bilinear restriction estimates for surfaces with curvatures of
  different signs},
        date={2006},
     journal={Transactions of the American Mathematical Society},
      volume={358},
      number={8},
       pages={3511\ndash 3533},
}

\bib{moyua1999restriction}{article}{
      author={Moyua, A.},
      author={Vargas, A.},
      author={Vega, L.},
       title={Restriction theorems and maximal operators related to oscillatory
  integrals in {$\mathbb R^3$}},
        date={1999},
     journal={Duke mathematical journal},
      volume={96},
      number={3},
       pages={547\ndash 574},
}

\bib{stein1979some}{inproceedings}{
      author={Stein, E.M.},
       title={Some problems in harmonic analysis},
        date={1979},
   booktitle={Harmonic analysis in euclidean spaces},
       pages={3\ndash 20},
}

\bib{stein1986oscillatory}{inproceedings}{
      author={Stein, E.M.},
       title={Oscillatory integrals in {Fourier} analysis},
        date={1986},
   booktitle={Beijing lectures in harmonic analysis},
       pages={307\ndash 355},
}

\bib{tao2001rotating}{article}{
      author={Tao, T.},
       title={From rotating needles to stability of waves: Emerging connections
  between combinatorics, analysis and {PDE}},
        date={2001},
     journal={Notices of the AMS},
      volume={48},
      number={3},
       pages={294\ndash 303},
}

\bib{tao2003sharp}{article}{
      author={Tao, T.},
       title={A sharp bilinear restriction estimate for paraboloids},
        date={2003},
     journal={Geometric and functional analysis},
      volume={13},
      number={6},
       pages={1359\ndash 1384},
}

\bib{tao2000bilinearI}{article}{
      author={Tao, T.},
      author={Vargas, A.},
       title={A bilinear approach to cone multipliers {I}. {Restriction}
  estimates},
        date={2000},
     journal={Geometric and functional analysis},
      volume={10},
      number={1},
       pages={185\ndash 215},
}

\bib{tao1998bilinear}{article}{
      author={Tao, T.},
      author={Vargas, A.},
      author={Vega, L.},
       title={A bilinear approach to the restriction and {Kakeya} conjectures},
        date={1998},
     journal={Journal of the American Mathematical Society},
      volume={11},
      number={4},
       pages={967\ndash 1000},
}

\bib{vargas2005restriction}{article}{
      author={Vargas, A.},
       title={Restriction theorems for a surface with negative curvature},
        date={2005},
     journal={Mathematische Zeitschrift},
      volume={249},
      number={1},
       pages={97\ndash 111},
}

\bib{wolff2001sharp}{article}{
      author={Wolff, T.},
       title={A sharp bilinear cone restriction estimate},
        date={2001},
     journal={Annals of Mathematics},
      volume={153},
      number={3},
       pages={661\ndash 698},
}

\end{biblist}
\end{bibdiv}


\end{document}